\documentclass{article}

\usepackage{amsfonts}
\usepackage{amsopn}
\usepackage{amsmath}
\usepackage{indentfirst,latexsym,bm}
\usepackage{amssymb}
\usepackage{amsmath}
\usepackage{graphics}
\usepackage{graphicx}
\usepackage{amsfonts}
\usepackage{amsmath}
\usepackage{indentfirst,latexsym}
\usepackage{amssymb}
\usepackage{amsmath}
\usepackage{graphics}
\usepackage{authblk}
\usepackage{color}
\usepackage{float}
\usepackage{latexsym}
\usepackage{booktabs}
\usepackage{multirow}
\usepackage{graphicx}
\usepackage{epsfig}
\usepackage{geometry}
\usepackage{amscd}
\usepackage{lscape}
\usepackage{float}
\usepackage{amsmath}
\usepackage{amssymb}
\usepackage{tabularx}
\usepackage{enumerate}
\usepackage{amsfonts}
\usepackage{mathrsfs}
\usepackage{eufrak}
\usepackage{graphicx}
\usepackage{dsfont}
\usepackage{epsfig}
\usepackage[ruled, lined, vlined, shortend, algosection]{algorithm2e}
\usepackage{longtable}
\usepackage{pdflscape}
\usepackage{lscape}
\usepackage{curves}
\usepackage{epic}
\usepackage{amsfonts}
\usepackage{amsmath}
\usepackage{indentfirst,latexsym,bm}
\usepackage{amssymb}
\usepackage{amsmath}
\usepackage{graphics}
\usepackage{authblk}
\usepackage{color}
\usepackage{float}
\usepackage{latexsym}
\usepackage{booktabs}
\usepackage{multirow}
\usepackage{graphicx}
\usepackage{epsfig}
\usepackage{geometry}
\usepackage{amscd}
\usepackage{lscape}
\usepackage{bbm}
\numberwithin{equation}{section}

\newtheorem{theorem}{Theorem}[section]

\newtheorem{assumption}[theorem]{Assumption}

\newtheorem{lemma}[theorem]{Lemma}

\newtheorem{proposition}[theorem]{Proposition}
\newtheorem{remark}[theorem]{Remark}

\newenvironment{proof}[1][Proof]{\textbf{#1.} }{\ \rule{0.5em}{0.5em}}
\DeclareMathOperator*{\esssup}{ess\,sup}

\DeclareMathOperator*{\essinf}{ess\,inf}
\DeclareMathOperator*{\Var}{\mbox{Var}}

\begin{document}

\title{{Risk-sensitive Dynkin games with heterogeneous Poisson random intervention times}
\thanks{We thank David Hobson for the suggestion of considering heterogenous signal times for constrained Dynkin games, which motivates the current project.}}

\author{\small\textsc{{Gechun Liang,\ \ Haodong Sun}}}
\affil{\small{\textsc{Department of Statistics, University of Warwick, Coventry, CV4 7AL, U.K.}}\\

\texttt{g.liang@warwick.ac.uk},\ \ \texttt{h.sun.9@warwick.ac.uk}}

\date{}
\maketitle

\begin{abstract} The paper solves constrained Dynkin games with
risk-sensitive criteria, where two players are allowed
to stop at two independent Poisson random intervention times, via
the theory of backward stochastic differential equations. This
generalizes the previous work of [Liang and Sun, Dynkin games with
Poisson random intervention times, SIAM Journal on Control and
Optimization, 2019] from the risk-neutral criteria and common
signal times for both players to the risk-sensitive
criteria and two heterogenous signal times. Furthermore, the
paper establishes a connection of such constrained risk-sensitive Dynkin games with
a class of stochastic differential games via Krylov's randomized
stopping technique.\\

\noindent\textit{Keywords}: Dynkin games, \and heterogenous Poisson
signal times,
\and backward stochastic differential equations, \and stochastic differential games, \and randomized stopping\\

\noindent\textit{Mathematics Subject Classification (2010)}: 91A15,
\and 91A55, \and 60H30.
\end{abstract}

%%%%%%%%%%%%%%%%%%%%%%%%%%%%%%%%%%%%%%%%%%%%%%%%%%%%%%%%%%%
%%%%%%%%%%%%%%%%%%%%%%%%%%%%%%%%%%%%%%%%%%%%%%%%%%%%%%%%%%

\section{Introduction}
Risk-sensitive criteria constitute a genuinely interesting class of
performance criteria in optimization problems, in which the linear
expectation $\mathbb{E}[\cdot]$ is replaced by the nonlinear
expectation
$$\tilde{\mathbbm{E}}\left[\cdot\right]:=g^{-1}(\mathbbm{E}\left[g(\cdot)\right]),$$
for some strictly increasing function $g$ as a risk-sensitive
function. The corresponding risk-sensitive control has been
developed to reflect an optimizer's attitudes to risks. In
particular, the risk-sensitive function $g$ is chosen to model the optimizer's
attitudes towards risks (e.g. strict concavity of $g$ reflects
risk-aversion {of maximization players or risk-seeking of minimization players}).

In this paper, we are interested in \emph{Dynkin games with risk-sensitive criteria},
by taking into account of both players' attitudes to risks. Namely, the two players aim
to minimize/maximize some payoff functional $R(\sigma,\tau)$ under the nonlinear expectation $\tilde{\mathbbm{E}}[\cdot]$:
\[J(\sigma,\tau)=\tilde{\mathbbm{E}}[R\left(\sigma,\tau\right)]=g^{-1}\left(\mathbbm{E}\left[g\left(R\left(\sigma,\tau\right)\right)\right]\right),\]
where $\sigma$ and $\tau$ are the stopping times to be chosen by the respective minimization/maximization players.
It is called risk-sensitive because
\[J(\sigma,\tau) \approx \mathbbm{E}[R(\sigma,\tau)]-\frac{1}{2}l_g\left(\mathbbm{E}[R(\sigma,\tau)]\right)\Var[R(\sigma,\tau)],\]
where $l_g(x)=-\frac{g''(x)}{g'(x)}$ is the Arrow-Pratt function of absolute risk aversion. The case $g(x)=x$ corresponds to a risk-neutral attitude of both players since $l_g(x)=0$. For the case of an exponential utility $g(x)=-e^{-\gamma x}$ with $\gamma>0$, $l_g(x)=\gamma$ is constant and the risk-sensitivity is only expressed through the risk-sensitivity parameter $\gamma$.

The stopping time strategies of the two players are restricted to two independent sequences of Poisson arrival times as
the exogenous constraints on the players' abilities to stop. The constraints may represent liquidity effects, indicating the times at which the underlying stochastic processes are available to stop. Applications of such a liquidity model can be found in
\cite{liang2015MAFE} for bank runs and \cite{liang2018dynkin} for convertible bonds. The constraints can also be seen as information constraints. The players are allowed to make their stopping decisions at all times, but they are only able to observe the underlying stochastic processes at Poisson arrival times. See \cite{dupuis2002optimal} and \cite{lempa2012optimal} for applications to perpetual American options. Due to the introduction of constraints on stopping times and risk-sensitive criteria, we call the Dynkin games considered in this paper the \emph{constrained risk-sensitive Dynkin games}.

We generalize our previous work \cite{liang2018dynkin} on \emph{constrained Dynkin games} in two aspects: First, it takes into consideration of both players' attitudes towards risks via the risk-sensitive function $g$; Second, there are control constraints for both players and, moreover, the constraints are different in the sense that they are allowed to stop at two heterogeneous sequences of Poisson arrival times. Consequently, since the two players' stopping time strategies are chosen from two different sequences of signal times, the usual condition of the upper obstacle $U$ dominating the lower one $L$ is not required.
In \cite{liang2018dynkin}, the risk-sensitive function $g(x)=x$ and both players stop at a single sequence of signal times (so $U\geq L$ is assumed therein).

New challenges arise from the above generalizations. Since the two players stop at two different sequences of Poisson arrival times, the first step to solve the constrained risk-sensitive Dynkin game is merging the two Poisson sequences together while still keeping track of their order. This is crucial when we consider a family of constrained risk-sensitive Dynkin games (\ref{upper_Values_discounted})-(\ref{lower_Values_discounted}) starting from different signal times in order to apply the dynamic programming principle. Note that the starting times of the games (\ref{upper_Values_discounted})-(\ref{lower_Values_discounted}) may not be the respective player's own Poisson signal times; instead they could be from the counterparty's signal times. To deal with the nonlinear expectation $\tilde{\mathbbm{E}}$ arising from the risk-sensitive function $g$, we introduce a new transformation resulting the auxiliary payoff processes (\ref{final_game_value_parameters_L})-(\ref{final_game_value_parameters_xi}), which enable us to rewrite the
payoff functional under the linear expectation $\mathbb{E}$ instead of the nonlinear expectation $\tilde{\mathbbm{E}}$.
For a special case of exponential risk-sensitive function $g$ (see
section \ref{section:exponential}), the representation formula (\ref{value_process}) of the game value
is closely related to Cole-Hopf transformation in the literature of backward stochastic differential equations (BSDEs for short), which is widely used to linearize a class of BSDEs with
quadratic growth (see \cite{Kobylanski}). Our representation formula
(\ref{value_process}) can be regarded as a stochastic control
version of Cole-Hopf transformation.

We also make a connection of constrained risk-sensitive Dynkin games with a class of stochastic differential games via
Krylov's randomized stopping technique (see \cite{Krylov}). It is established in \cite{Krylov} that standard optimal stopping problems (without constraints on stopping times) admit stochastic control representation, which can be further solved via the so-called normalized Bellman equations. The stochastic control representation of the corresponding constrained
optimal stopping problems has been established in \cite{liang2015stochastic} (see section
4 therein). In a constrained stopping game setting as considered in the current paper, it is natural to expect that a stochastic differential game representation should hold accordingly. Indeed, we show that the two players in the stochastic differential game choose their respective running controls and discount rates with binary values $0$ or the Poisson intensity $\lambda^i$, and the optimal control is the Poisson intensity $\lambda^i$ whenever the value of the game falls below the lower obstacle process/goes above the upper obstacle process.

Turing to the literature of Dynkin games,
there has been a considerable development
since the seminal works of Dynkin \cite{dynkin1969game} and Neveu \cite{neveu1975discrete}. The continuous time models were developed,
among others, by Bismut \cite{bismut1977probleme}, Alario-Nazaret et
al \cite{alario1982dynkin}, Lepeltier and Maingueneau
\cite{lepeltier1984jeu} and Morimoto \cite{morimoto1984dynkin}. In
order to relax the dominating condition $U\geq L$ in those papers, Yasuda \cite{yasuda1985randomized}
proposed the strategies of randomized stopping times, and proved that the game
value exists under merely an integrability condition. Rosemberg et
al \cite{rosenberg2001stopping}, Touzi and Vielle
\cite{touzi2002continuous} and Laraki and Solan
\cite{laraki2005value} further extended his work in this direction. The non-Markovian case was addressed in  Cvitanic and Karatzas \cite{cvitanic1996backward} for a fixed horizon and Hamadene et al \cite{HLW1999} for an infinite horizon using the theory of reflected BSDEs.
If the two players in the game are with asymmetric payoffs/information, then it
gives arise to a nonzero-sum Dynkin game. See, for example, Hamadene
and Zhang \cite{hamadene2010}, De Angelis et al
\cite{De Angelis2015} and, more recently, De Angelis and Ekstrom \cite{De Angelis2020} with more references therein. A robust version
of Dynkin games can be found in Bayraktar and Yao
\cite{Bayraktar2017} if the players are ambiguous about their
probability model.

On the other hand, the risk-sensitive optimal stopping problems have been studied by Nagai \cite{nagai2007stopping}, B\"{a}uerle and Rieder \cite{baauerle2017partially}, B\"{a}uerle and Popp \cite{bauerle2018risk} and, more recently, Jelito et al \cite{jelito2019risk}.  For the risk-sensitive zero-sum and nonzero-sum stochastic differential games, we refer to El-Karoui and Hamad{\`e}ne \cite{el2003bsdes}.
To the best of our knowledge, the study of risk-sensitive Dynkin games is still lacking, no matter with or without constraints on stopping time strategies. The current paper offers a first step to understand risk-sensitive Dynkin games, in particular with constraints on the stopping time strategies.

The constrained optimal stopping problems was first studied by Dupuis and Wang \cite{dupuis2002optimal}, where they used it to model perpetual American options exercised at exogenous Poisson arrival times. See also Lempa \cite{lempa2012optimal}, Menaldi and Robin \cite{menaldi2016some} and Hobson and Zeng \cite{HZ} for further extensions of this type of optimal stopping models. From a different perspective, Liang \cite{liang2015stochastic} made a connection between such kind of optimal stopping problems with penalized BSDEs. The corresponding optimal switching (impulse control) models were studied by Liang and Wei \cite{liang2013optimal}, and by Menaldi and Robin \cite{menaldi2017some} \cite{menaldi2018some} with  more general signal times and state spaces. More recently, Liang and Sun \cite{liang2018dynkin} introduced the corresponding constrained Dynkin games (with the risk-sensitive function $g(x)=x$), where both players were allowed to stop at a sequence of random times generated by a single exogenous Poisson process serving as a signal process.

The paper is organized as follows. Section 2 contains the problem formulation and main result, with its proof provided in section 3. In section 4, we establish its connection with a class of stochastic differential games, and in section 5 we further provide two examples. Finally, section 6 concludes the paper.
%%%%%%%%%%%%%%%%%%%%%%%%%%%%%%%%%%%%%%%%%%%%%%%%%%%%%%%%%%%%%%%%%%%%%%%
\section{Constrained risk-sensitive Dynkin games}
Let $(\Omega,\mathcal{F},\mathbbm{P})$ be a complete probability space endowed with a $d$-dimensional standard Brownian motion
$(W_t)_{t\geq 0}$ with $\mathbbm{F}=\{\mathcal{F}_t\}_{t\geq 0}$ being the minimal augmented filtration of $W$.
The probability space also supports two independent sequences of Poisson arrival times $T^{(1)}=\{T^{(1)}_n\}_{n\geq 0}$ and $T^{(2)}=\{T^{(2)}_n\}_{n\geq 0}$ with their respective intensities $\lambda^{(1)}$ and $\lambda^{(2)}$ and the minimal augmented filtration $\mathbbm{H}=\{\mathcal{H}_t\}_{t\geq 0}$, satisfying $T^{(1)}_0=T^{(2)}_0=0$ and $T^{(1)}_{\infty}=T^{(2)}_{\infty}=\infty$. Denote the smallest filtration generated by $\mathbbm{F}$ and $\mathbbm{H}$ as $\mathbbm{G}=\{\mathcal{G}_t\}_{t\geq 0}$, i.e. $\mathcal{G}_t=\mathcal{F}_t\vee\mathcal{H}_t$, and write $\boldsymbol{\lambda}=(\lambda^{(1)},\lambda^{(2)})$.

Let $T$ be a finite $\mathbbm{F}$-stopping time representing the
(random) terminal time of the game. For each player $i\in\{1,2\}$, let us define a random variable $M_i:\Omega\mapsto
\mathbbm{N}$ such that $T_{M_i}$ is the next arrival time in the Poisson sequence $T^{(i)}$ following $T$, i.e. $M_i(\omega):= \sum_{n\geq 1}n\mathbbm{1}_{\{T^i_{n-1}(\omega)\leq T(\omega)<T^i_{n}(\omega)\}}$.
%where $\Omega_0:=\{\omega\in\Omega:T(\omega)=+\infty\}$, with $\mathbbm{P}(\Omega_0)\in [0,1]$. Three typical examples are (1) fixed horizon $[0,T]$ so $\mathbb{P}(\Omega_0)=0$; (2) infinite horizon $[0,\infty)$ so $\mathbb{P}(\Omega_0)=1$; (3) random horizon $[0,T]$ with $T$ as a first passage time so $\mathbb{P}(\Omega_0)\in(0,1)$ (see section 6 in \cite{liang2018dynkin}).

For any integer $n\geq 0$, let us define the control set for each player
$i\in\{1,2\}$ as
\begin{equation}\label{control_set}
\mathcal{R}_n^{(i)}=\{\mathbbm{G}\mbox{-stopping time }\sigma\mbox{
for }\sigma(\omega)=T^{(i)}_N(\omega)\mbox{ where }n\leq N\leq
M_i(\omega)\},
\end{equation}
so the player $i$ chooses from the
Poisson arrival times $T^{(i)}$ with intensity $\lambda^{(i)}$, and
$T_n^{(i)}$ is the smallest stopping time allowed.

Consider a \emph{constrained risk-sensitive Dynkin game}, where the two players choose their respective stopping times $\sigma\in\mathcal{R}_1^{(1)}$ and $\tau\in\mathcal{R}_1^{(2)}$ in order to minimize/maximize the expected cost functional
\begin{equation}
\label{risk_sensitive_cost_functional}
J(\sigma,\tau)=\tilde{\mathbbm{E}}\left[R(\sigma,\tau)\right],
\end{equation}
where the nonlinear expectation $\tilde{\mathbbm{E}}:\mathbbm{R} \to \mathbbm{R}$ is defined via the risk-sensitive function $g$, i.e.
\begin{equation}
\label{def_nonlinear_expectation}
\tilde{\mathbbm{E}}\left[\cdot\right] :=g^{-1}\left(\mathbbm{E}\left[g\left(\cdot\right)\right]\right).
\end{equation}
The discounted payoff functional $R(\sigma,\tau)$ in (\ref{risk_sensitive_cost_functional}) is defined by
\begin{equation}
\label{def_payoff}
R(\sigma,\tau)=\int_0^{\sigma\wedge\tau\wedge
T}e^{-rs}f_s\,ds+e^{-rT}\xi\mathbbm{1}_{\{ \sigma\wedge \tau \geq T\}}+e^{-r\tau}L_{\tau}\mathbbm{1}_{\{\tau<T,\tau\leq\sigma\}}+e^{-r\sigma}U_{\sigma}\mathbbm{1}_{\{\sigma<T,\sigma<
\tau\}},
\end{equation}
where $r>0$ is the discount rate, and $f$, as a real-valued
$\mathbbm{F}$-progressively measurable process, is the running
payoff. The terminal payoff is $U$ if $\sigma$ happens firstly, {$L$
if $\tau$ happens firstly or $\sigma$ and $\tau$ happen
simultaneously}, and $\xi$ otherwise, where $L$ and $U$ are two
real-valued $\mathbbm{F}$-progressively measurable processes, and
$\xi$ is a real-valued $\mathcal{F}_T$-measurable random variable.
%We take the convention that if no players ever stop (i.e.
%$\tau=\sigma=\infty$), the payoff is null. Note that the above
%formulation (\ref{def_payoff}) also covers the perpetual case by
%taking $T=\infty$:
%$$\int_0^{\sigma\wedge\tau}e^{-rs}f_s\,ds\mathbbm{1}_{\{\sigma\wedge \tau<\infty\}}+e^{-r\tau}L_{\tau}\mathbbm{1}_{\{\tau<\infty,\tau\leq\sigma\}}+e^{-r\sigma}U_{\sigma}\mathbbm{1}_{\{\sigma<\infty,\sigma<
%\tau\}}.$$

Let us define the upper and lower values of the constrained risk-sensitve Dynkin game
\begin{equation}
\label{upper_lowerValues}
\overline{v}^{\boldsymbol{\lambda}}=\inf_{\sigma\in\mathcal{R}_1^{(1)}}\sup_{\tau\in\mathcal{R}_1^{(2)}}J(\sigma,\tau),\mbox{ and }\underline{v}^{\boldsymbol{\lambda}}=\sup_{\tau\in\mathcal{R}_1^{(2)}}\inf_{\sigma\in\mathcal{R}_1^{(1)}}J(\sigma,\tau).
\end{equation}
The game (\ref{upper_lowerValues}) is said to have value $v^{\boldsymbol{\lambda}}$ if $v^{\boldsymbol{\lambda}}=\overline{v}^{\boldsymbol{\lambda}}=\underline{v}^{\boldsymbol{\lambda}}$, and a saddle point $(\sigma^{*},\tau^{*})\in \mathcal{R}_1^{(1)}\times \mathcal{R}_1^{(2)}$ is called an optimal stopping strategy of the game if
\[J(\sigma^{*},\tau) \leq J(\sigma^{*},\tau^{*}) \leq J(\sigma,\tau^{*}),\]
for every $(\sigma,\tau)\in \mathcal{R}_1^{(1)}\times \mathcal{R}_1^{(2)}$.

Compared with the constrained Dynkin game introduced in \cite{liang2018dynkin}, there are two new features of the game (\ref{upper_lowerValues}): First, it takes into consideration of the both players' attitudes towards risks via the risk-sensitive function $g$; Second, there are control constraints for both players and, moreover, the constraints are different in the sense that they are allowed to stop at two heterogeneous sequences of Poisson arrival times. As a consequence, since the two players' stopping time strategies are chosen from two different control sets, the usual dominating condition $U\geq L$ is not required.
In \cite{liang2018dynkin}, the risk-sensitive function $g(x)=x$ and both players stop at a single sequence of Poisson arrival times (so $U\geq L$ is a critical assumption therein).

\subsection{Main result}
To solve the above constrained risk-sensitive Dynkin game, we introduce the characterizing BSDE on a random horizon $[0,T]$:
{\begin{equation}
\label{final_game_value}
\overline{Q}_{t\wedge T}^{\boldsymbol{\lambda}}=\overline{\xi}+ \int_{t\wedge T}^T\left[-\lambda^{(1)} \left(\overline{Q}_s^{\boldsymbol{\lambda}}-\overline{U}_s\right)^+ + \lambda^{(2)}\left(\overline{L}_s-\overline{Q}_s^{\boldsymbol{\lambda}} \right)^+-r\overline{Q}_{s}^{\boldsymbol{\lambda}}\right]\,ds - \int_{t\wedge T}^T\overline{Z}_s^{\boldsymbol{\lambda}}\,dW_s,
\end{equation}}
for $t \geq 0$, where the \emph{auxiliary payoff processes} $\overline{L}$, $\overline{U}$ and $\overline{\xi}$ are given by
\begin{eqnarray}
\overline{L}_t &=& e^{rt}g(e^{-rt}L_t+\int_0^t  e^{-ru}f_u\,du),\label{final_game_value_parameters_L}\\
\overline{U}_t &=& e^{rt}g(e^{-rt}U_t+\int_0^t e^{-ru}f_u\,du),\label{final_game_value_parameters_U}\\
\overline{\xi} &=& e^{rT}g(e^{-rT}\xi+\int_0^T e^{-ru}f_u\,du),\label{final_game_value_parameters_xi}
\end{eqnarray}
respectively. And also we set $\overline{Q}_{t}^{\boldsymbol{\lambda}}\equiv \overline{\xi}$ for $t\geq T$. Moreover, we introduce the following spaces: for any
given $\alpha\in\mathbbm{R}$ and $n\in\mathbbm{N}$,
\begin{itemize}
\item $\mathbbm{L}^{2,n}_{\alpha}: \mathcal{F}_T$-measurable random variables $\xi:\Omega \mapsto \mathbbm{R}^n$ with $\mathbbm{E}\left[e^{2\alpha T}||\xi||^2\right]<\infty$,
\item $\mathbbm{H}^{2,n}_{\alpha}:\mathbbm{F}$-progressively measurable processes $\varphi :[0,T]\times\Omega \mapsto \mathbbm{R}^n$ with $\mathbbm{E}\left[\int_0^Te^{2\alpha s}||\varphi_s||^2\,ds\right]<\infty$,
\item $\mathbbm{S}^{2,n}_{\alpha}:\mathbbm{F}$-progressively measurable processes $\varphi :[0,T]\times\Omega \mapsto \mathbbm{R}^n$ with $\mathbbm{E}\left[\sup_{s\in[0,T]}e^{2\alpha s}||\varphi_s||^2\right]<\infty$,
\end{itemize}
where we denote $\mathbbm{L}^{2,n}_{0}$, $\mathbbm{H}^{2,n}_{0}$ and
$\mathbbm{S}^{2,n}_{0}$ by $\mathbbm{L}^{2,n}$, $\mathbbm{H}^{2,n}$
and $\mathbbm{S}^{2,n}$ for the ease of notation.

We impose the following assumptions on the risk-sensitive function
$g$, the running payoff $f$ and the terminal payoffs $L$, $U$ and
$\xi$ in terms of the auxiliary payoffs $\overline{L}$, $\overline{U}$ and
$\overline{\xi}$.

\begin{assumption}
\label{assumption_1}
The deterministic risk-sensitive function $g:\mathbbm{R}\to\mathbbm{R}$ is strictly increasing and, moreover,
(i) when $T$ is an unbounded stopping time, $\overline{L}$, $\overline{U}$ and $\overline{\xi}$ are all bounded; (ii) when $T$ is a bounded stopping time, $\overline{L}\in\mathbbm{S}^{2,1}$, $\overline{U}\in\mathbbm{S}^{2,1}$ and $\overline{\xi}\in\mathbbm{L}^{2,1}$, where $\overline{L}$, $\overline{U}$ and $\overline{\xi}$ are given by (\ref{final_game_value_parameters_L}), (\ref{final_game_value_parameters_U}) and (\ref{final_game_value_parameters_xi}), respectively.
\end{assumption}

On one hand, since the two players' control sets are different, the usual dominating condition $U\geq L$ is not required. On the other hand,
the conditions (i) and (ii) in Assumption \ref{assumption_1} guarantee the existence and uniqueness of the solution to BSDE
(\ref{final_game_value}), which will in turn be used to construct the game value and its associated optimal stopping strategy. Under Assumption \ref{assumption_1}, the solvability of BSDE (\ref{final_game_value}) follows from Theorem 3.3 in \cite{briand1998stability} (when $T$ is unbounded) and Theorem 4.1 in \cite{pardoux1990adapted} (when $T$ is bounded), and thus we omit the proof of the following proposition and refer to \cite{briand1998stability} and \cite{pardoux1990adapted} for the details.

\begin{proposition}
\label{prop_bsde_solvability}
Suppose that Assumption \ref{assumption_1} holds. Then, there exists a unique solution $(\overline{Q}^{\boldsymbol{\lambda}},\overline{Z}^{\boldsymbol{\lambda}})$ to BSDE (\ref{final_game_value}). Moreover, (i) when $T$ is an unbounded stopping time, $\overline{Q}^{\boldsymbol{\lambda}}$ is continuous and bounded, and $\overline{Z}^{\boldsymbol{\lambda}}$ belongs to $\mathbbm{H}^{2,d}_{-r}$; (ii) when $T$ is a bounded stopping time, the solution pair $(\overline{Q}^{\boldsymbol{\lambda}},\overline{Z}^{\boldsymbol{\lambda}})$ belong to $\mathbbm{S}^{2,1} \times\mathbbm{H}^{2,d}$.
\end{proposition}

We are now in a position to state the main result of this paper.

\begin{theorem}
\label{bigTheorem} Suppose that Assumption \ref{assumption_1} holds. Let $(\overline{Q}^{\boldsymbol{\lambda}},\overline{Z}^{\boldsymbol{\lambda}})$ be the unique solution to BSDE (\ref{final_game_value}), and define the value process
\begin{equation}\label{value_process}
Q_t^{\boldsymbol{\lambda}} =
e^{r (t\wedge T)}g^{-1}(e^{-r(t\wedge T)}\overline{Q}^{\boldsymbol{\lambda}}_t)-\int_0^{t\wedge T}
e^{-r(u-t\wedge T)}f_u\,du,
\end{equation}
for $t\geq 0$. Then, the value of the constrained risk-sensitive Dynkin game (\ref{upper_lowerValues}) exists and is given by
\[v^{\boldsymbol{\lambda}}=\overline{v}^{\boldsymbol{\lambda}}=\underline{v}^{\boldsymbol{\lambda}}=Q_0^{\boldsymbol{\lambda}}.\]
Moreover, the optimal stopping strategy of the game is given by
\begin{equation*}
\left\{\begin{array}{l}
\sigma^{*}=\inf\{T_N^{(1)}\geq T^{(1)}_1:Q_{T_N^{(1)}}^{\boldsymbol{\lambda}}\geq U_{T_N^{(1)}}\}\wedge T^{(1)}_{M_1};\\
\tau^{*}=\inf\{T_N^{(2)}\geq
T^{(2)}_1:Q_{T_N^{(2)}}^{\boldsymbol{\lambda}}\leq
L_{T_N^{(2)}}\}\wedge T^{(2)}_{M_2}.
\end{array}\right.
\end{equation*}
\end{theorem}
\begin{remark}
For a special case of exponential risk-sensitive function $g$ (see
section \ref{section:exponential}), the representation formula
(\ref{value_process}) is closely related to Cole-Hopf transformation in the BSDE
literature, which is widely used to linearize a class of BSDEs with
quadratic growth (see \cite{Kobylanski}). Our representation formula
(\ref{value_process}) can be regarded as a stochastic control
version of Cole-Hopf transformation.
\end{remark}

%%%%%%%%%%%%%%%%%%%%%%%%%%%%%%%%%%%%%%%%%%%%%%%%%%%%%%%%%%%%%%%%%%%%%%%
\section{Proof of Theorem \ref{bigTheorem}}
Since the two players stop at two different sequences of Poisson arrival times, the first step to prove Theorem \ref{bigTheorem} is merging the two Poisson sequences together while still keeping track of their order.
To this end, for each $T^{(1)}$ and $T^{(2)}$, we construct an increasing sequence of $\mathbbm{G}$-stopping times $\theta=(\theta_k)_{k\geq 0}$ as follows:
\begin{align*}
\theta_0 &= T^{(1)}_0=T^{(2)}_0=0,\\
\theta_1 &= \min\left(T^{(1)}_1,T^{(2)}_1\right),\\
\theta_2 &= \min\left(T^{(1)}_1\mathbbm{1}_{\{T^{(1)}_1>\theta_1\}}+T^{(1)}_2\mathbbm{1}_{\{T^{(1)}_1\leq \theta_1\}},T^{(2)}_1\mathbbm{1}_{\{T^{(2)}_1>\theta_1\}}+T^{(2)}_2\mathbbm{1}_{\{T^{(2)}_1\leq \theta_1\}}\right),\\
\theta_3 &=\min\left(T^{(1)}_1\mathbbm{1}_{\{T^{(1)}_1>\theta_2\}}+T^{(1)}_3\mathbbm{1}_{\{T^{(1)}_1\leq \theta_2\}},T^{(1)}_2\mathbbm{1}_{\{T^{(1)}_2>\theta_2\}}+T^{(1)}_3\mathbbm{1}_{\{T^{(1)}_2\leq \theta_2\}},\right.\\
&\qquad\qquad \left.T^{(2)}_1\mathbbm{1}_{\{T^{(2)}_1>\theta_2\}}+T^{(2)}_3\mathbbm{1}_{\{T^{(2)}_1\leq \theta_2\}},T^{(2)}_2\mathbbm{1}_{\{T^{(2)}_2>\theta_2\}}+T^{(2)}_3\mathbbm{1}_{\{T^{(2)}_2\leq \theta_2\}}\right),\\
&\cdots, \\
\theta_k &= \min\left(T^{(1)}_1\mathbbm{1}_{\{T^{(1)}_1>\theta_{k-1}\}}+T^{(1)}_k\mathbbm{1}_{\{T^{(1)}_1\leq \theta_{k-1}\}},\cdots,T^{(1)}_{k-1}\mathbbm{1}_{\{T^{(1)}_{k-1}>\theta_{k-1}\}}+T^{(1)}_k\mathbbm{1}_{\{T^{(1)}_{k-1}\leq \theta_{k-1}\}},\right.\\
&\qquad\qquad \left.T^{(2)}_1\mathbbm{1}_{\{T^{(2)}_1>\theta_{k-1}\}}+T^{(2)}_k\mathbbm{1}_{\{T^{(2)}_1\leq \theta_{k-1}\}},\cdots,T^{(2)}_{k-1}\mathbbm{1}_{\{T^{(2)}_{k-1}>\theta_{k-1}\}}+T^{(2)}_k\mathbbm{1}_{\{T^{(2)}_{k-1}\leq \theta_{k-1}\}}\right),\\
&\cdots.
\end{align*}

In Figure \ref{sequence_simulation}, we illustrate the construction of the merged sequence $\theta$, where the top and the middle line are a realization of $T^{(1)}$ and $T^{(2)}$, and the bottom line is the merged sequence $\theta$. Intuitively, given any $\mathbbm{G}$-stopping time $\theta_{k-1}$, $k\geq 1$, (to be used as the starting times for a family of constrained Dynkin games (\ref{upper_Values_discounted})-(\ref{lower_Values_discounted}) below), we find the first arrival time of each Poisson sequence following $\theta_{k-1}$, say $T^{(1)}_{k_1}$ and $T^{(2)}_{k_2}$ for some $k_1,k_2\geq 0$, and then define $\theta_{k}=\min\{T^{(1)}_{k_1},T^{(2)}_{k_2}\}$.
%As in (\ref{def_M_i}), we define a random variable $M:\Omega\mapsto
%\overline{\mathbbm{N}}:=\mathbbm{N} \cup \{+\infty\}$ such that
%\begin{equation}
%\label{def_M}
%M(\omega):= \left\{\begin{array}{ll}
%\sum_{n\geq 1}i\mathbbm{1}_{\{\theta_{n-1}(\omega)\leq T(\omega)<\theta_{n}(\omega)\}},&\mbox{ if }\omega \in \Omega\backslash\Omega_0,\\
%+\infty,&\mbox{ if }\omega \in \Omega_0,
%\end{array}\right.\end{equation}
%where $\Omega_0:=\{\omega\in\Omega:T(\omega)=+\infty\}$.
Moreover, given the stopping time $\theta_k$, we define pre-$\theta_k$ $\sigma$-field:
\[\mathcal{G}_{\theta_k}=\left\{A\in \bigvee_{s\geq 0}\mathcal{G}_s:A\cap \{\theta_k\leq s\}\in \mathcal{G}_s\mbox{ for }s\geq 0\right\},\]
and $\tilde{\mathbbm{G}}=\{\mathcal{G}_{\theta_k}\}_{k\geq 0}$.

\begin{figure}[h]
  \includegraphics[width=\textwidth]{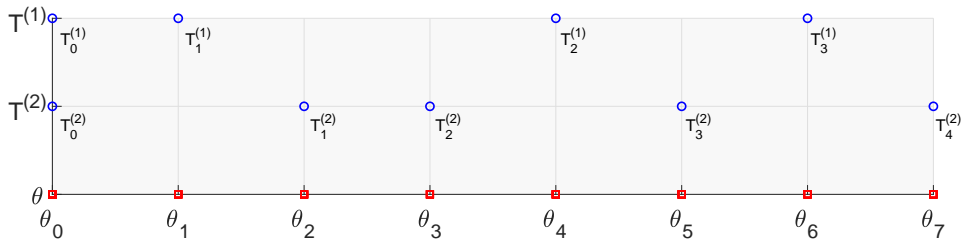}
  \centering
  \caption{An illustration of a merged Poisson arrival sequence $\theta$.}
  \label{sequence_simulation}
\end{figure}

Next, we tackle the nonlinear expectation $\tilde{\mathbbm{E}}$ associated with the risk-sensitive function $g$. To this end, introduce the \emph{discounted processes}
\begin{eqnarray}
\tilde{L}_t = e^{-rt}L_t+\int_0^t  e^{-ru}f_u\,du,\label{final_game_value_parameters_discounted_L}\\
\tilde{U}_t = e^{-rt}U_t+\int_0^t e^{-ru}f_u\,du,\label{final_game_value_parameters_discounted_U}\\
\tilde{\xi} = e^{-rT}\xi+\int_0^T e^{-ru}f_u\,du,\label{final_game_value_parameters_discounted_xi}
\end{eqnarray}
and rewrite the discounted payoff functional $R(\sigma,\tau)$ as
\begin{equation}
\label{def_payoff_modified}
\tilde{R}(\sigma,\tau)=\tilde{\xi}\mathbbm{1}_{\{\sigma\wedge \tau \geq T\}}+\tilde{L}_{\tau}\mathbbm{1}_{\{\tau<T,\tau\leq\sigma\}}+\tilde{U}_{\sigma}\mathbbm{1}_{\{\sigma<T,\sigma< \tau\}}=R(\sigma,\tau).
\end{equation}
In turn, consider a \emph{family of constrained risk-sensitive Dynkin games} starting from $\theta_{k-1}$, for $k\geq 1$, whose upper and lower values are defined by
\begin{eqnarray}
\overline{q}_{\theta_{k-1}}^{\boldsymbol{\lambda}}=\essinf_{\sigma\in\tilde{\mathcal{R}}^{(1)}_{\theta_k}}\esssup_{\tau\in\tilde{\mathcal{R}}^{(2)}_{\theta_k}} \tilde{\mathbbm{E}}\left[\tilde{R}\left(\sigma,\tau\right)|\mathcal{G}_{\theta_{k-1}}\right],\label{upper_Values_discounted}\\
\underline{q}_{\theta_{k-1}}^{\boldsymbol{\lambda}}=\esssup_{\tau\in\tilde{\mathcal{R}}^{(2)}_{\theta_k}}\essinf_{\sigma\in \tilde{\mathcal{R}}^{(1)}_{\theta_k}}  \tilde{\mathbbm{E}}\left[\tilde{R}\left(\sigma,\tau\right)|\mathcal{G}_{\theta_{k-1}}\right],\label{lower_Values_discounted}
\end{eqnarray}
where
\begin{equation}
\label{def_control_set_modified}
\tilde{\mathcal{R}}_{\theta_{k}}^{(i)}=\{\mathbbm{G}\mbox{-stopping
time }\sigma\mbox{ for }\sigma(\omega)=T^{(i)}_N(\omega)\mbox{ where
}T^{(i)}_N(\omega)\geq \theta_{k}\mbox{ and }N\leq M_i(\omega)\}.
\end{equation}

\begin{remark}
Note that in the above definition of control set
$\tilde{\mathcal{R}}^{(i)}_{\theta_k}$, $\theta_{k}$ is not
necessary from the Poisson sequence $T^{(i)}$, so
$\tilde{\mathcal{R}}^{(i)}_{\theta_k}$ is in general different from
$\mathcal{R}_k^{(i)}$ in (\ref{control_set}). However, they do
coincide when $k=1$:
$\tilde{\mathcal{R}}^{(i)}_{\theta_1}=\mathcal{R}_1^{(i)}$.

On the other hand, thanks to the introduction of the discounted
processes $\tilde{L}$, $\tilde{U}$ and $\tilde{\xi}$ in
(\ref{final_game_value_parameters_discounted_L})-(\ref{final_game_value_parameters_discounted_xi}),
the payoff functional in (\ref{def_payoff_modified}) can be divided
into three disjoint sets and the risk-sensitive function $g$ can be
applied to each of them separately. Thus, we can rewrite the
payoff in
(\ref{upper_Values_discounted})-(\ref{lower_Values_discounted})
under the linear expectation $\mathbb{E}$ of the auxiliary payoff processes $\overline{L}$, $\overline{U}$ and $\overline{\xi}$ as
$$\tilde{\mathbbm{E}}\left[\tilde{R}\left(\sigma,\tau\right)|\mathcal{G}_{\theta_{k-1}}\right]=g^{-1}\left(\mathbbm{E}\left[
e^{-rT}\overline{\xi}\mathbbm{1}_{\{ \sigma\wedge \tau \geq T\}}+e^{-r\tau}\overline{L}_{\tau}\mathbbm{1}_{\{\tau<T,\tau\leq\sigma\}}+e^{-r\sigma}\overline{U}_{\sigma}\mathbbm{1}_{\{\sigma<T,\sigma< \tau\}}|\mathcal{G}_{\theta_{k-1}}\right]\right).
$$
This motives us to introduce the Cole-Hopf representation formula (\ref{value_process}).
\end{remark}

The constrained risk-sensitive Dynkin game (\ref{upper_Values_discounted})-(\ref{lower_Values_discounted}) is said to have value $q_{\theta_{k-1}}^{\boldsymbol{\lambda}}$ if $q_{\theta_{k-1}}^{\boldsymbol{\lambda}}=\overline{q}_{\theta_{k-1}}^{\boldsymbol{\lambda}}=\underline{q}_{\theta_{k-1}}^{\boldsymbol{\lambda}}$, and $(\sigma^*_{k},\tau^*_k)\in\tilde{\mathcal{R}}^{(1)}_{\theta_k} \times \tilde{\mathcal{R}}^{(2)}_{\theta_k}$ is called an optimal stopping strategy of the game if
\[\tilde{\mathbbm{E}}\left[\tilde{R}(\sigma^*_{k},\tau)|\mathcal{G}_{\theta_{k-1}}\right] \leq \tilde{\mathbbm{E}}\left[\tilde{R}(\sigma^*_{k},\tau^*_{k})|\mathcal{G}_{\theta_{k-1}}\right] \leq \tilde{\mathbbm{E}}\left[\tilde{R}(\sigma,\tau^*_{k})|\mathcal{G}_{\theta_{k-1}}\right],\]
for every $(\sigma,\tau)\in\tilde{\mathcal{R}}^{(1)}_{\theta_k} \times \tilde{\mathcal{R}}^{(2)}_{\theta_k}$.
In particular,  when $k = 1$, (\ref{upper_Values_discounted})-(\ref{lower_Values_discounted}) corresponds to the original constrained Dynkin game (\ref{upper_lowerValues}). Thus, to prove Theorem \ref{bigTheorem}, it is sufficient to show that
\[q_{\theta_{k-1}}^{\boldsymbol{\lambda}}=\overline{q}_{\theta_{k-1}}^{\boldsymbol{\lambda}}=\underline{q}_{\theta_{k-1}}^{\boldsymbol{\lambda}}=\tilde{Q}_{\theta_{k-1}}^{\boldsymbol{\lambda}},\]
and the optimal stopping strategy is given by
\begin{equation}
\label{optimalStop_general}
\left\{\begin{array}{l}
\sigma^{*}_{k}=\inf\{T^{(1)}_N\geq \theta_{k}:\tilde{Q}_{T^{(1)}_N}^{\boldsymbol{\lambda}} \geq \tilde{U}_{T^{(1)}_N}\}\wedge T^{(1)}_{M_1},\\
\tau^{*}_{k}=\inf\{T^{(2)}_N\geq
\theta_{k}:\tilde{Q}_{T^{(2)}_N}^{\boldsymbol{\lambda}} \leq
\tilde{L}_{T^{(2)}_N}\}\wedge T^{(2)}_{M_2},
\end{array}\right.
\end{equation}
where $\tilde{Q}^{\boldsymbol{\lambda}}$ is given by
\begin{equation}
\label{final_game_value_general_discounted}
\tilde{Q}^{\boldsymbol{\lambda}}_t =
g^{-1}(e^{-r(t\wedge T)}\overline{Q}^{\boldsymbol{\lambda}}_t),
\end{equation}
with $\overline{Q}^{\boldsymbol{\lambda}}$ being the first component
of the solution to BSDE (\ref{final_game_value}). In turn, the value
process $Q^{\boldsymbol{\lambda}}$ in (\ref{bigTheorem}) is given
via the discounted process $\tilde{Q}^{\boldsymbol{\lambda}}$ via
the relationship
\begin{equation}\label{value_process_relation}
Q^{\boldsymbol{\lambda}}_t=e^{r(t\wedge T)}\tilde{Q}_t^{\boldsymbol{\lambda}}-\int_0^{t\wedge T}e^{-r(u-t\wedge T)}f_udu.
\end{equation}
Note that, for $t\geq T$,
$$Q^{\boldsymbol{\lambda}}_t=e^{r T}g^{-1}(e^{-r T}\overline{\xi})-\int_0^{ T}e^{-r(u-T)}f_udu=\xi.$$

\begin{remark} For the reader's convenience, we recall the notations that have been introduced
thus far. For the payoff processes $h=L,U,\xi$, we have defined the discounted processes
$\tilde{h}_t=e^{-rt}h_t+\int_0^t e^{-ru}f_u\,du,$ and auxiliary payoff processes
$\overline{h}_t=e^{rt}g(\tilde{h}_t)$. In terms of the value process
$Q^{\boldsymbol{\lambda}}$, likewise we have
$\tilde{Q}^{\boldsymbol{\lambda}}_t=e^{-rt}Q^{\boldsymbol{\lambda}}_t+\int_0^t
e^{-ru}f_u\,du$, and
$\overline{Q}^{\boldsymbol{\lambda}}_t=e^{rt}g(\tilde{Q}^{\boldsymbol{\lambda}}_t)$, for $t\in[0,T]$.
\end{remark}

To prove the above assertions (and therefore Theorem \ref{bigTheorem}), we start with the following lemma.
\begin{lemma}
\label{firstLemma} Suppose that Assumption \ref{assumption_1} holds.
Then, $\tilde{Q}_{\theta_{k-1}}^{\boldsymbol{\lambda}}$ given in
(\ref{final_game_value_general_discounted}) satisfies the recursive
equation
\begin{multline}
\label{recursiveEq} \tilde{Q}_{\theta_{k-1}}^{\boldsymbol{\lambda}}=
\tilde{\mathbbm{E}}\left[\left.\tilde{\xi}\mathbbm{1}_{\{\theta_k \geq T\}}+\left(\min\{\tilde{U}_{\theta_k},\tilde{Q}^{\boldsymbol{\lambda}}_{\theta_k}
\}\mathbbm{1}_{\{\theta_k \in
T^{(1)}\}}+\max\{\tilde{L}_{\theta_k},\tilde{Q}_{\theta_k}^{\boldsymbol{\lambda}}\}\mathbbm{1}_{\{\theta_k
\in T^{(2)} \}}\right)\mathbbm{1}_{\{\theta_k <
T\}}\right|\mathcal{G}_{\theta_{k-1}}\right],
\end{multline}
for $k\geq 1 $.
\end{lemma}

\begin{proof}
It is equivalent to prove that
\begin{multline}
\label{recursiveEq_exponential}
g(\tilde{Q}^{\boldsymbol{\lambda}}_{\theta_{k-1}})=\
\mathbbm{E}\left[g(\tilde{\xi})\mathbbm{1}_{\{\theta_k
\geq T\}}\right.\\
\left.\left.+\left(\min\{g(\tilde{U}_{\theta_k}),g(\tilde{Q}^{\boldsymbol{\lambda}}_{\theta_k})
\}\mathbbm{1}_{\{\theta_k \in
T^{(1)}\}}+\max\{g(\tilde{L}_{\theta_k}),g(\tilde{Q}_{\theta_k}^{\boldsymbol{\lambda}})\}\mathbbm{1}_{\{\theta_k
\in T^{(2)} \}}\right)\mathbbm{1}_{\{\theta_k <
T\}}\right|\mathcal{G}_{\theta_{k-1}}\right],
\end{multline}
where $g(\tilde{\xi}) = e^{-rT}\overline{\xi}$, $g(\tilde{L}_t) =
e^{-rt}\overline{L}_t$ and $g(\tilde{U}_t) = e^{-rt}\overline{U}_t$. {For $k$ such that $\theta_{k-1} > T$, it follows from (\ref{final_game_value_general_discounted}) that $g(\tilde{Q}_{\theta_{k-1}}^{\boldsymbol{\lambda}})=g(\tilde{\xi})$, and thus (\ref{recursiveEq_exponential}) holds. In the rest of the proof, we only focus on the cases where $\theta_{k-1} \leq T$.}

By applying It\^{o}'s formula to $\alpha_t g(\tilde{Q}_t^{\boldsymbol{\lambda}})$, where $\alpha_t=e^{-(\lambda^{(1)}+\lambda^{(2)})t}$, we can obtain that
\begin{align*}
\alpha_{t\wedge T} g(\tilde{Q}^{\boldsymbol{\lambda}}_{t\wedge T}) =&\ \alpha_T g(\tilde{\xi}) +\int_{t\wedge T}^T \alpha_s\bigg[(\lambda^{(1)}+\lambda^{(2)})g(\tilde{Q}^{\boldsymbol{\lambda}}_s)-\lambda^{(1)} \left(g(\tilde{Q}^{\boldsymbol{\lambda}}_s)-g(\tilde{U}_s)\right)^+\\
&+\lambda^{(2)}\left(g(\tilde{L}_s)-g(\tilde{Q}^{\boldsymbol{\lambda}}_s)\right)^+\bigg]\,ds-\int_{t\wedge T}^T \alpha_se^{-rs}\overline{Z}_s^{\boldsymbol{\lambda}}\,dW_s\\
=&\ \alpha_T g(\tilde{\xi})+\int_{t\wedge T}^T \alpha_s\left[\lambda^{(1)} \min\left\{g(\tilde{U}_s),g(\tilde{Q}^{\boldsymbol{\lambda}}_s)\right\}+\lambda^{(2)}\max\left\{g(\tilde{L}_s),g(\tilde{Q}^{\boldsymbol{\lambda}}_s)\right\}\right]\,ds\\
&-\int_{t\wedge T}^T \alpha_se^{-rs}\overline{Z}_s^{\boldsymbol{\lambda}}\,dW_s,
\end{align*}
for $t \geq 0$. By choosing $t=\theta_{k-1}$ and taking the conditional expectation with respect to $\mathcal{G}_{\theta_{k-1}}$, we further have
\begin{multline}
\label{proof_recursive_equation_eq01}
g(\tilde{Q}^{\boldsymbol{\lambda}}_{\theta_{k-1}}) = \mathbbm{E}\bigg[e^{-(\lambda^{(1)}+\lambda^{(2)})(T-\theta_{k-1})}g(\tilde{\xi})\\
+\int_{\theta_{k-1}}^T e^{-(\lambda^{(1)}+\lambda^{(2)})(s-\theta_{k-1})}\left[\lambda^{(1)} \min\left\{g(\tilde{U}_s),g(\tilde{Q}^{\boldsymbol{\lambda}}_s)\right\}+\lambda^{(2)}\max\left\{g(\tilde{L}_s),g(\tilde{Q}^{\boldsymbol{\lambda}}_s)\right\}\right]\,ds\bigg|\mathcal{G}_{\theta_{k-1}}\bigg],
\end{multline}
for any $k \geq 1$.

On the other hand, by defining $\tilde{T}^{(i)}_t$ as the first arrival time in $T^{(i)}$ following any fixed time $t$, i.e. $\tilde{T}^{(i)}_t=\inf\{T_N^{(i)} \geq T_1^{(i)}: T_N^{(i)}> t\}$, we can rewrite the right-hand-side of (\ref{recursiveEq_exponential}) as
\begin{multline}
\label{recursiveEq_exponential_rewrite}
\mathbbm{E} \left[g(\tilde{\xi})\mathbbm{1}_{\{\tilde{T}^{(1)}_{\theta_{k-1}} \wedge \tilde{T}^{(2)}_{\theta_{k-1}} \geq T\}}+\min\left\{g(\tilde{U}_{\tilde{T}^{(1)}_{\theta_{k-1}}}),g(\tilde{Q}_{\tilde{T}^{(1)}_{\theta_{k-1}}}^{\boldsymbol{\lambda}})\right\}\mathbbm{1}_{\{\tilde{T}^{(1)}_{\theta_{k-1}}< T,\tilde{T}^{(1)}_{\theta_{k-1}}<\tilde{T}^{(2)}_{\theta_{k-1}}\}}\right.\\
\qquad\qquad\qquad +\left.\left.\max\left\{g(\tilde{L}_{\tilde{T}^{(2)}_{\theta_{k-1}}}),g(\tilde{Q}_{\tilde{T}^{(2)}_{\theta_{k-1}}}^{\boldsymbol{\lambda}})\right\}\mathbbm{1}_{\{\tilde{T}^{(2)}_{\theta_{k-1}}< T,\tilde{T}^{(2)}_{\theta_{k-1}} \leq \tilde{T}^{(1)}_{\theta_{k-1}}\}}\right|\mathcal{G}_{\theta_{k-1}}\right].
\end{multline}
Indeed, applying the joint probability density function of $(\tilde{T}^{(1)}_{\theta_{k-1}},\tilde{T}^{(2)}_{\theta_{k-1}})$ conditional on $\mathcal{G}_{\theta_{k-1}}$,
\[p_{\theta_{k-1}}(S,U)=\lambda^{(1)}e^{-\lambda^{(1)}(S-\theta_{k-1})}\lambda^{(2)}e^{-\lambda^{(2)}(U-\theta_{k-1})},\]
yields that
\begin{align*}
&\mathbbm{E} \left[\left.g(\tilde{\xi})\mathbbm{1}_{\{\tilde{T}^{(1)}_{\theta_{k-1}} \wedge \tilde{T}^{(2)}_{\theta_{k-1}} \geq T\}}\right|\mathcal{G}_{\theta_{k-1}}\right] \\
=&\ \mathbbm{E} \left[\left.g(\tilde{\xi})\iint_{S\wedge U \geq T}p_{\theta_{k-1}}(S,U)\,dS\,dU\right|\mathcal{G}_{\theta_{k-1}}\right]\\
=&\ \mathbbm{E} \bigg[g(\tilde{\xi})\underbrace{\iint_{U \geq S \geq T}\lambda^{(1)}e^{-\lambda^{(1)}(S-\theta_{k-1})}\lambda^{(2)}e^{-\lambda^{(2)}(U-\theta_{k-1})}\,dS\,dU}_{\mbox{(I)}}\bigg|\mathcal{G}_{\theta_{k-1}}\bigg]\\
&+ \mathbbm{E} \bigg[g(\tilde{\xi})\underbrace{\iint_{S \geq U \geq T} \lambda^{(1)}e^{-\lambda^{(1)}(S-\theta_{k-1})}\lambda^{(2)}e^{-\lambda^{(2)}(U-\theta_{k-1})} \,dS\,dU}_{\mbox{(II)}}\bigg|\mathcal{G}_{\theta_{k-1}}\bigg],
\end{align*}
where the first integral
\[\mbox{(I)} = \lambda^{(1)} \int_T^{\infty} e^{-\lambda^{(1)}(S-\theta_{k-1})}(\int_S^{\infty} \lambda^{(2)}e^{-\lambda^{(2)}(U-\theta_{k-1})}\,dU )\,dS = \frac{\lambda^{(1)}}{\lambda^{(1)}+\lambda^{(2)}}e^{-(\lambda^{(1)}+\lambda^{(2)})(T-\theta_{k-1})},\]
and, similarly, the second integral
\[\mbox{(II)} = \frac{\lambda^{(2)}}{\lambda^{(1)}+\lambda^{(2)}}e^{-(\lambda^{(1)}+\lambda^{(2)})(T-\theta_{k-1})}.\]
In turn, we obtain
\begin{equation}
\label{eq_3.17_rhs_eq01}
\mathbbm{E} \left[\left.g(\tilde{\xi})\mathbbm{1}_{\{\tilde{T}^{(1)}_{\theta_{k-1}} \wedge \tilde{T}^{(2)}_{\theta_{k-1}} \geq T\}}\right|\mathcal{G}_{\theta_{k-1}}\right] = \mathbbm{E}\left[\left.e^{-(\lambda^{(1)}+\lambda^{(2)})(T-\theta_{k-1})}g(\tilde{\xi})\right|\mathcal{G}_{\theta_{k-1}}\right].
\end{equation}
Similarly, we have
\begin{align}
&\mathbbm{E} \left[\left.\min\left\{g(\tilde{U}_{\tilde{T}^{(1)}_{\theta_{k-1}}}),g(\tilde{Q}_{\tilde{T}^{(1)}_{\theta_{k-1}}}^{\boldsymbol{\lambda}})\right\}\mathbbm{1}_{\{\tilde{T}^{(1)}_{\theta_{k-1}}< T,\tilde{T}^{(1)}_{\theta_{k-1}}<\tilde{T}^{(2)}_{\theta_{k-1}}\}} \right|\mathcal{G}_{\theta_{k-1}}\right] \nonumber\\
=&\ \mathbbm{E} \left[\left.\iint_{\theta_{k-1}<S< T,S<U}\min\left\{g(\tilde{U}_S),g(\tilde{Q}_S^{\boldsymbol{\lambda}})\right\}p_{\theta_{k-1}}(S,U)\,dS\,dU \right|\mathcal{G}_{\theta_{k-1}}\right] \nonumber\\
=&\ \mathbbm{E} \left[\left.\int_{\theta_{k-1}}^{T}\lambda^{(1)}e^{-\lambda^{(1)}(S-\theta_{k-1})}\min\left\{g(\tilde{U}_S),g(\tilde{Q}_S^{\boldsymbol{\lambda}})\right\}(\int_S^{\infty}\lambda^{(2)}e^{-\lambda^{(2)}(U-\theta_{k-1})}\,dU)\,dS \right|\mathcal{G}_{\theta_{k-1}}\right] \nonumber\\
=&\ \mathbbm{E} \left[\left.\int_{\theta_{k-1}}^{T}\lambda^{(1)}e^{-(\lambda^{(1)}+\lambda^{(2)})(S-\theta_{k-1})}\min\left\{g(\tilde{U}_S),g(\tilde{Q}_S^{\boldsymbol{\lambda}})\right\}\,dS \right|\mathcal{G}_{\theta_{k-1}}\right], \label{eq_3.17_rhs_eq02}
\end{align}
and
\begin{align}
&\mathbbm{E} \left[\left.\max\left\{g(\tilde{L}_{\tilde{T}^{(2)}_{\theta_{k-1}}}),g(\tilde{Q}_{\tilde{T}^{(2)}_{\theta_{k-1}}}^{\boldsymbol{\lambda}})\right\}\mathbbm{1}_{\{\tilde{T}^{(2)}_{\theta_{k-1}}< T,\tilde{T}^{(2)}_{\theta_{k-1}} \leq \tilde{T}^{(1)}_{\theta_{k-1}}\}} \right|\mathcal{G}_{\theta_{k-1}}\right]\nonumber\\
=&\ \mathbbm{E} \left[\left.\int_{\theta_{k-1}}^{T}\lambda^{(2)}e^{-(\lambda^{(1)}+\lambda^{(2)})(U-\theta_{k-1})}\max\left\{g(\tilde{L}_U),g(\tilde{Q}_U^{\boldsymbol{\lambda}})\right\}\,dU \right|\mathcal{G}_{\theta_{k-1}}\right].\label{eq_3.17_rhs_eq03}
\end{align}

It follows from (\ref{proof_recursive_equation_eq01}), (\ref{eq_3.17_rhs_eq01}), (\ref{eq_3.17_rhs_eq02}) and (\ref{eq_3.17_rhs_eq03}) that (\ref{recursiveEq_exponential}) holds for any $k \geq 1$. Hence, $Q_{\theta_{k-1}}^{\boldsymbol{\lambda}}$, which is given by (\ref{final_game_value_general_discounted}), satisfies the recursive equation (\ref{recursiveEq}), for $k\geq 1$.
\end{proof}

As a direct consequence of Lemma \ref{firstLemma}, we deduce that $\hat{Q}_{\theta_{k-1}}^{\boldsymbol{\lambda}}$ defined by
{\begin{equation}
\label{def_Q_hat}
\hat{Q}_{\theta_{k-1}}^{\boldsymbol{\lambda}}:=\tilde{\xi}\mathbbm{1}_{\{\theta_{k-1} \geq T\}}+\bigg(\min\{\tilde{U}_{\theta_{k-1}},\tilde{Q}_{\theta_{k-1}}^{\boldsymbol{\lambda}}\}\mathbbm{1}_{\{\theta_{k-1}\in T^{(1)}\}}+\max\{\tilde{L}_{\theta_{k-1}},\tilde{Q}_{\theta_{k-1}}^{\boldsymbol{\lambda}}\}\mathbbm{1}_{\{\theta_{k-1}\in T^{(2)}\}}\bigg)\mathbbm{1}_{\{\theta_{k-1} < T\}},
\end{equation}}
where $\tilde{Q}_{\theta_{k-1}}^{\boldsymbol{\lambda}}$ is given by (\ref{final_game_value_general_discounted}), satisfies the recursive equation
\begin{multline}
\label{recursive_equ_system}
\hat{Q}_{\theta_{k-1}}^{\boldsymbol{\lambda}} = \tilde{\xi}\mathbbm{1}_{\{\theta_{k-1} \geq T\}}+\bigg(\min\{\tilde{U}_{\theta_{k-1}},\tilde{\mathbbm{E}}[\hat{Q}^{\boldsymbol{\lambda}}_{\theta_k}|\mathcal{G}_{\theta_{k-1}}]\}\mathbbm{1}_{\{\theta_{k-1}\in T^{(1)}\}}\\
+\max\{\tilde{L}_{\theta_{k-1}},\tilde{\mathbbm{E}}[\hat{Q}^{\boldsymbol{\lambda}}_{\theta_k}|\mathcal{G}_{\theta_{k-1}}]\}\mathbbm{1}_{\{\theta_{k-1}\in T^{(2)}\}}\bigg)\mathbbm{1}_{\{\theta_{k-1} < T\}},
\end{multline}
for $k \geq 1$.

We will show that $\hat{Q}_{\theta_{k-1}}^{\boldsymbol{\lambda}}$ in (\ref{def_Q_hat}) is actually the unique solution of the recursive equation (\ref{recursive_equ_system}). The uniqueness is proved by showing that $\hat{Q}_{\theta_{k-1}}^{\boldsymbol{\lambda}}$ is the value of an \emph{auxiliary constrained risk-sensitive Dynkin game} starting from $\theta_{k-1}$, whose upper and lower values are defined by
\begin{eqnarray}
\overline{\hat{q}}_{\theta_{k-1}}^{\boldsymbol{\lambda}}=\essinf_{\sigma\in \tilde{\mathcal{R}}_{\theta_{k-1}}^{(1)}}\esssup_{\tau\in \tilde{\mathcal{R}}_{\theta_{k-1}}^{(2)}}\tilde{\mathbbm{E}}\left[\tilde{R}(\sigma,\tau)|\mathcal{G}_{\theta_{k-1}}\right],\label{aux_upperValues2}\\
\underline{\hat{q}}_{\theta_{k-1}}^{\boldsymbol{\lambda}}=\esssup_{\tau\in \tilde{\mathcal{R}}_{\theta_{k-1}}^{(2)}}\essinf_{\sigma\in \tilde{\mathcal{R}}_{\theta_{k-1}}^{(1)}}\tilde{\mathbbm{E}}\left[\tilde{R}(\sigma,\tau)|\mathcal{G}_{\theta_{k-1}}\right],\label{aux_lowerValues2}
\end{eqnarray}
where the payoff functional $\tilde{R}(\sigma,\tau)$ is given by (\ref{def_payoff_modified}) and the control set $\tilde{\mathcal{R}}^{(i)}_{\theta_{k-1}}$ is given by (\ref{def_control_set_modified}).

The auxiliary game (\ref{aux_upperValues2})-(\ref{aux_lowerValues2}) is said to have value $\hat{q}_{\theta_{k-1}}^{\boldsymbol{\lambda}}$ if $\hat{q}_{\theta_{k-1}}^{\boldsymbol{\lambda}}=\overline{\hat{q}}_{\theta_{k-1}}^{\boldsymbol{\lambda}}=\underline{\hat{q}}_{\theta_{k-1}}^{\boldsymbol{\lambda}}$, and $(\hat{\sigma}^{*}_{k-1},\hat{\tau}^{*}_{k-1})\in\tilde{\mathcal{R}}_{\theta_{k-1}}^{(1)}\times\tilde{\mathcal{R}}_{\theta_{k-1}}^{(2)}$ is called an optimal stopping strategy of the game (\ref{aux_upperValues2})-(\ref{aux_lowerValues2}) if
\[\tilde{\mathbbm{E}}\left[\tilde{R}(\sigma^*_{k-1},\tau)|\mathcal{G}_{\theta_{k-1}}\right] \leq \tilde{\mathbbm{E}}\left[\tilde{R}(\sigma^*_{k-1},\tau^*_{k-1})|\mathcal{G}_{\theta_{k-1}}\right] \leq \tilde{\mathbbm{E}}\left[\tilde{R}(\sigma,\tau^*_{k-1})|\mathcal{G}_{\theta_{k-1}}\right],\]
for every $(\sigma,\tau)\in\tilde{\mathcal{R}}_{\theta_{k-1}}^{(1)}\times\tilde{\mathcal{R}}_{\theta_{k-1}}^{(2)}$.

The difference between (\ref{aux_upperValues2})-(\ref{aux_lowerValues2}) and (\ref{upper_Values_discounted})-(\ref{lower_Values_discounted}) is that the players first make their stopping decisions and then move
forward in the former game, while in the latter game they first move forward and then make their decisions.

\begin{lemma}
\label{secondLemma} Suppose that Assumption \ref{assumption_1} holds. Then, for any $k\geq 1$, the value of the auxiliary constrained risk-sensitive Dynkin game (\ref{aux_upperValues2})-(\ref{aux_lowerValues2}) starting from $\theta_{k-1}$  exists. Its value $\hat{q}_{\theta_{k-1}}^{\boldsymbol{\lambda}}$ is the unique solution of the recursive equation (\ref{recursive_equ_system}). Hence, $\hat{q}^{\boldsymbol{\lambda}}_{\theta_{k-1}}=\hat{Q}^{\boldsymbol{\lambda}}_{\theta_{k-1}}$, where the latter is given by (\ref{def_Q_hat}). The optimal stopping strategy of the auxiliary constrained risk-sensitive Dynkin game (\ref{aux_upperValues2})-(\ref{aux_lowerValues2}) is given by
\begin{equation}
\label{optimalStop_aux}
\left\{\begin{array}{l}
\hat{\sigma}^{*}_{k-1}=\inf\{T^{(1)}_N\geq \theta_{k-1}:\hat{Q}_{T^{(1)}_N}^{\boldsymbol{\lambda}}= \tilde{U}_{T^{(1)}_N}\}\wedge T^{(1)}_{M_1};\\
\hat{\tau}^{*}_{k-1}=\inf\{T^{(2)}_N\geq \theta_{k-1}:\hat{Q}_{T^{(2)}_N}^{\boldsymbol{\lambda}}= \tilde{L}_{T^{(2)}_N}\}\wedge T^{(2)}_{M_2}.
\end{array}\right.
\end{equation}
\end{lemma}

\begin{proof}
\noindent\emph{Step 1.} Let $\hat{Q}_{\theta_{k-1}}^{\boldsymbol{\lambda}}$ be a solution of the recursive equation (\ref{recursive_equ_system}) for $k\geq 1$. We claim the following martingale properties hold:

%\noindent(i) $\left(\tilde{\xi}\mathbbm{1}_{\{\theta_{m}\wedge
%\hat{\sigma}^*_{k-1}\wedge
%\hat{\tau}^*_{k-1} \geq T\}}+\hat{Q}_{\theta_{m}\wedge
%\hat{\sigma}^*_{k-1}\wedge
%\hat{\tau}^*_{k-1}}^{\boldsymbol{\lambda}}\mathbbm{1}_{\{\theta_{m}\wedge
%\hat{\sigma}^*_{k-1}\wedge
%\hat{\tau}^*_{k-1} < T\}}\right)_{m \geq k-1}$ is a
%$\tilde{\mathbbm{G}}$-martingale under the nonlinear expectation $\tilde{\mathbbm{E}}$;
%
%\noindent(ii) $\left(\tilde{\xi}\mathbbm{1}_{\{ \theta_{m}\wedge
%\hat{\sigma}^*_{k-1}\wedge
%\tau \geq T\}}+\hat{Q}_{\theta_{m}\wedge
%\hat{\sigma}^*_{k-1}\wedge
%\tau}^{\boldsymbol{\lambda}}\mathbbm{1}_{\{\theta_{m}\wedge
%\hat{\sigma}^*_{k-1}\wedge
%\tau < T\}}\right)_{m \geq k-1}$ is a $\tilde{\mathbbm{G}}$-supermartingale under the nonlinear expectation $\tilde{\mathbbm{E}}$, for any
%$\tau\in\tilde{\mathcal{R}}_{\theta_{k-1}}^{(2)}$;
%
%\noindent(iii) $\left(\tilde{\xi}\mathbbm{1}_{\{\theta_{m}\wedge
%\sigma\wedge
%\hat{\tau}^*_{k-1}  \geq T\}}+\hat{Q}_{\theta_{m}\wedge
%\sigma\wedge
%\hat{\tau}^*_{k-1}}^{\boldsymbol{\lambda}}\mathbbm{1}_{\{\theta_{m}\wedge
%\sigma\wedge
%\hat{\tau}^*_{k-1} < T\}}\right)_{m \geq k-1}$ is a
%$\tilde{\mathbbm{G}}$-submartingale under the nonlinear expectation $\tilde{\mathbbm{E}}$, for any
%$\sigma\in\tilde{\mathcal{R}}_{\theta_{k-1}}^{(1)}$.

\noindent(i) $\left(\hat{Q}_{\theta_{m}\wedge
\hat{\sigma}^*_{k-1}\wedge
\hat{\tau}^*_{k-1}}^{\boldsymbol{\lambda}}\right)_{m \geq k-1}$ is a
$\tilde{\mathbbm{G}}$-martingale under the nonlinear expectation $\tilde{\mathbbm{E}}$;

\noindent(ii) $\left(\hat{Q}_{\theta_{m}\wedge
\hat{\sigma}^*_{k-1}\wedge
\tau}^{\boldsymbol{\lambda}}\right)_{m \geq k-1}$ is a $\tilde{\mathbbm{G}}$-supermartingale under $\tilde{\mathbbm{E}}$, for any
$\tau\in\tilde{\mathcal{R}}_{\theta_{k-1}}^{(2)}$;

\noindent(iii) $\left(\hat{Q}_{\theta_{m}\wedge
\sigma\wedge
\hat{\tau}^*_{k-1}}^{\boldsymbol{\lambda}}\right)_{m \geq k-1}$ is a
$\tilde{\mathbbm{G}}$-submartingale under $\tilde{\mathbbm{E}}$, for any
$\sigma\in\tilde{\mathcal{R}}_{\theta_{k-1}}^{(1)}$.

If the martingale property (i) holds, then, for $k \geq 1$,
%\begin{align*}
%\hat{Q}_{\theta_{k-1}}^{\boldsymbol{\lambda}}
%=&\ \hat{Q}_{\theta_{k-1}\wedge
%\hat{\sigma}^*_{k-1}\wedge
%\hat{\tau}^*_{k-1}}^{\boldsymbol{\lambda}}\\
%=&\ \tilde{\xi}\mathbbm{1}_{\{\theta_{k-1}\wedge
%\hat{\sigma}^*_{k-1}\wedge
%\hat{\tau}^*_{k-1}  \geq T\}}+\hat{Q}_{\theta_{k-1}\wedge
%\hat{\sigma}^*_{k-1}\wedge
%\hat{\tau}^*_{k-1}}^{\boldsymbol{\lambda}}\mathbbm{1}_{\{\theta_{k-1}\wedge
%\hat{\sigma}^*_{k-1}\wedge
%\hat{\tau}^*_{k-1} < T\}}\\
%=&\ \tilde{\mathbbm{E}}\left[\left.\tilde{\xi}\mathbbm{1}_{\{\hat{\sigma}^*_{k-1}\wedge
%\hat{\tau}^*_{k-1}  \geq T\}}+\hat{Q}_{\hat{\sigma}^*_{k-1}\wedge
%\hat{\tau}^*_{k-1}}^{\boldsymbol{\lambda}}\mathbbm{1}_{\{\hat{\sigma}^*_{k-1}\wedge
%\hat{\tau}^*_{k-1} < T\}}\right|\mathcal{G}_{\theta_{k-1}}\right].
%\end{align*}
\[\hat{Q}_{\theta_{k-1}}^{\boldsymbol{\lambda}}=\ \hat{Q}_{\theta_{k-1}\wedge
\hat{\sigma}^*_{k-1}\wedge
\hat{\tau}^*_{k-1}}^{\boldsymbol{\lambda}}=\tilde{\mathbbm{E}}\left[\hat{Q}_{\hat{\sigma}^*_{k-1}\wedge
\hat{\tau}^*_{k-1}}^{\boldsymbol{\lambda}}|\mathcal{G}_{\theta_{k-1}}\right],\]
and the definition of $(\hat{\sigma}^*_{k-1},\hat{\tau}^*_{k-1})$ in (\ref{optimalStop_aux}) further yields that
\begin{align}
\hat{Q}_{\theta_{k-1}}^{\boldsymbol{\lambda}}
=&\ \tilde{\mathbbm{E}}\bigg[\tilde{\xi}\mathbbm{1}_{\{\hat{\sigma}^*_{k-1}\wedge
\hat{\tau}^*_{k-1}  \geq T\}}+\hat{Q}_{\hat{\tau}^*_{k-1}}^{\boldsymbol{\lambda}}\mathbbm{1}_{\{\hat{\tau}^*_{k-1} < T,\hat{\tau}^*_{k-1} \leq \hat{\sigma}^*_{k-1}\}}+\hat{Q}_{\hat{\sigma}^*_{k-1}}^{\boldsymbol{\lambda}}\mathbbm{1}_{\{\hat{\sigma}^*_{k-1}< T,\hat{\sigma}^*_{k-1} < \hat{\tau}^*_{k-1} \}}|\mathcal{G}_{\theta_{k-1}}\bigg]\nonumber\\
=&\ \tilde{\mathbbm{E}}\bigg[\tilde{\xi}\mathbbm{1}_{\{\hat{\sigma}^*_{k-1}\wedge
\hat{\tau}^*_{k-1}  \geq T\}}+\tilde{L}_{\hat{\tau}^*_{k-1}}\mathbbm{1}_{\{\hat{\tau}^*_{k-1} < T,\hat{\tau}^*_{k-1} \leq \hat{\sigma}^*_{k-1}\}}+\tilde{U}_{\hat{\sigma}^*_{k-1}}\mathbbm{1}_{\{\hat{\sigma}^*_{k-1}< T,
\hat{\sigma}^*_{k-1} < \hat{\tau}^*_{k-1} \}}|\mathcal{G}_{\theta_{k-1}}\bigg]\nonumber\\
=&\ \tilde{\mathbbm{E}}\bigg[\tilde{R}(\hat{\sigma}^*_{k-1},\hat{\tau}^*_{k-1})|\mathcal{G}_{\theta_{k-1}}\bigg].\label{main_proof_step1_eq01}
\end{align}

Using the similar arguments, if the supermartingale property (ii) and the submartingale property (iii) hold, then we have, for any $\tau\in\tilde{\mathcal{R}}_{\theta_{k-1}}^{(2)}$,
\begin{align}
\hat{Q}_{\theta_{k-1}}^{\boldsymbol{\lambda}}
 \geq &\ \tilde{\mathbbm{E}}\bigg[\hat{Q}_{\hat{\sigma}^*_{k-1}\wedge
\tau}^{\boldsymbol{\lambda}}|\mathcal{G}_{\theta_{k-1}}\bigg]\nonumber\\
=&\ \tilde{\mathbbm{E}}\bigg[\tilde{\xi}\mathbbm{1}_{\{\hat{\sigma}^*_{k-1}\wedge
\tau  \geq T\}}+\hat{Q}_{\tau}^{\boldsymbol{\lambda}}\mathbbm{1}_{\{\tau < T,\tau \leq \hat{\sigma}^*_{k-1}\}}+\hat{Q}_{\hat{\sigma}^*_{k-1}}^{\boldsymbol{\lambda}}\mathbbm{1}_{\{\hat{\sigma}^*_{k-1}< T,\hat{\sigma}^*_{k-1} < \tau \}}|\mathcal{G}_{\theta_{k-1}}\bigg]\nonumber\\
\geq &\ \tilde{\mathbbm{E}}\bigg[\tilde{\xi}\mathbbm{1}_{\{\hat{\sigma}^*_{k-1}\wedge
\tau  \geq T\}}+\tilde{L}_{\tau}\mathbbm{1}_{\{\tau < T,\tau \leq \hat{\sigma}^*_{k-1}\}}+\tilde{U}_{\hat{\sigma}^*_{k-1}}\mathbbm{1}_{\{\hat{\sigma}^*_{k-1}< T,\hat{\sigma}^*_{k-1} < \tau \}}|\mathcal{G}_{\theta_{k-1}}\bigg]\nonumber\\
=&\ \tilde{\mathbbm{E}}\bigg[\tilde{R}(\hat{\sigma}^*_{k-1},\tau)|\mathcal{G}_{\theta_{k-1}}\bigg],\label{main_proof_step1_eq02}
\end{align}
and, for any $\sigma\in\tilde{\mathcal{R}}_{\theta_{k-1}}^{(1)}$,
\begin{equation}
\label{main_proof_step1_eq03}
\hat{Q}_{\theta_{k-1}}^{\boldsymbol{\lambda}} \leq \tilde{\mathbbm{E}}\bigg[\tilde{R}(\sigma,\hat{\tau}^*_{k-1})|\mathcal{G}_{\theta_{k-1}}\bigg].
\end{equation}

It follows from (\ref{main_proof_step1_eq02}) and (\ref{main_proof_step1_eq03}) that
\[\hat{Q}_{\theta_{k-1}}^{\boldsymbol{\lambda}} \geq \esssup_{\tau\in\tilde{\mathcal{R}}_{\theta_{k-1}}^{(2)}}\tilde{\mathbbm{E}}\bigg[\tilde{R}(\hat{\sigma}^*_{k-1},\tau)|\mathcal{G}_{\theta_{k-1}}\bigg]\geq \essinf_{\sigma\in\tilde{\mathcal{R}}_{\theta_{k-1}}^{(1)}}\esssup_{\tau\in\tilde{\mathcal{R}}_{\theta_{k-1}}^{(2)}}\tilde{\mathbbm{E}}\bigg[\tilde{R}(\sigma,\tau)|\mathcal{G}_{\theta_{k-1}}\bigg] =\overline{\hat{q}}_{\theta_{k-1}}^{\boldsymbol{\lambda}},\]
and
\[\hat{Q}_{\theta_{k-1}}^{\boldsymbol{\lambda}} \leq\essinf_{\sigma\in\tilde{\mathcal{R}}_{\theta_{k-1}}^{(1)}}\tilde{\mathbbm{E}}\bigg[\tilde{R}(\sigma,\hat{\tau}^*_{k-1})|\mathcal{G}_{\theta_{k-1}}\bigg]\leq \esssup_{\tau\in\tilde{\mathcal{R}}_{\theta_{k-1}}^{(2)}}\essinf_{\sigma\in\tilde{\mathcal{R}}_{\theta_{k-1}}^{(1)}}\tilde{\mathbbm{E}}\bigg[\tilde{R}(\sigma,\tau)|\mathcal{G}_{\theta_{k-1}}\bigg] = \underline{\hat{q}}_{\theta_{k-1}}^{\boldsymbol{\lambda}}.\]
It is clear that $\overline{\hat{q}}_{\theta_{k-1}}^{\boldsymbol{\lambda}} \geq \underline{\hat{q}}_{\theta_{k-1}}^{\boldsymbol{\lambda}}$, and therefore the value of the auxiliary constrained risk-sensitive Dynkin game (\ref{aux_upperValues2})-(\ref{aux_lowerValues2}) exists, i.e.
\[\hat{Q}_{\theta_{k-1}}^{\boldsymbol{\lambda}}=\hat{q}_{\theta_{k-1}}^{\boldsymbol{\lambda}}=\overline{\hat{q}}_{\theta_{k-1}}^{\boldsymbol{\lambda}}=\underline{\hat{q}}_{\theta_{k-1}}^{\boldsymbol{\lambda}}.\]
This also implies the recursive equation (\ref{recursive_equ_system}) admits a unique solution. Furthermore, since $\hat{Q}_{\theta_{k-1}}^{\boldsymbol{\lambda}}$ given by (\ref{def_Q_hat}) satisfies the recursive equation (\ref{recursive_equ_system}), it is actually the unique solution of (\ref{recursive_equ_system}). As a direct consequence of (\ref{main_proof_step1_eq01})-(\ref{main_proof_step1_eq03}), we can obtain that $(\hat{\sigma}^*_{k-1},\hat{\tau}^*_{k-1})$, which is given by (\ref{optimalStop_aux}), is indeed an optimal stopping strategy of the auxiliary constrained risk-sensitive Dynkin game (\ref{aux_upperValues2})-(\ref{aux_lowerValues2}).

\noindent\emph{Step 2.} It remains to prove the martingale property (i), the supermartingale property (ii) and the submartingale property (iii) in Step 1.

Indeed, for $m\geq k-1$, we have
\begin{align*}
\tilde{\mathbbm{E}}\left[\left.\hat{Q}_{\theta_{m+1}\wedge\hat{\sigma}^*_{k-1}\wedge\hat{\tau}^*_{k-1}}^{\boldsymbol{\lambda}}\right|\mathcal{G}_{\theta_m}\right] = &\ \tilde{\mathbbm{E}}\bigg[\mathbbm{1}_{\{\hat{\sigma}^*_{k-1}\wedge \hat{\tau}^*_{k-1} \leq \theta_m\}} \hat{Q}_{\hat{\sigma}^*_{k-1}\wedge\hat{\tau}^*_{k-1}}^{\boldsymbol{\lambda}}+\mathbbm{1}_{\{\hat{\sigma}^*_{k-1}\wedge \hat{\tau}^*_{k-1} \geq \theta_{m+1}\}} \hat{Q}_{\theta_{m+1}}^{\boldsymbol{\lambda}}|\mathcal{G}_{\theta_m}\bigg]\\
= &\ \mathbbm{1}_{\{\hat{\sigma}^*_{k-1}\wedge \hat{\tau}^*_{k-1} \leq \theta_m\}} \hat{Q}_{\hat{\sigma}^*_{k-1}\wedge\hat{\tau}^*_{k-1}}^{\boldsymbol{\lambda}} +\mathbbm{1}_{\{\hat{\sigma}^*_{k-1}\wedge \hat{\tau}^*_{k-1} \geq \theta_{m+1}\}} \tilde{\mathbbm{E}}[\hat{Q}_{\theta_{m+1}}^{\boldsymbol{\lambda}} |\mathcal{G}_{\theta_m}]\\
= &\ \mathbbm{1}_{\{\hat{\sigma}^*_{k-1}\wedge \hat{\tau}^*_{k-1} \leq \theta_m\}} \hat{Q}_{\hat{\sigma}^*_{k-1}\wedge\hat{\tau}^*_{k-1}}^{\boldsymbol{\lambda}} +\mathbbm{1}_{\{\hat{\sigma}^*_{k-1}\wedge \hat{\tau}^*_{k-1} \geq \theta_{m+1}\}} \hat{Q}_{\theta_{m}}^{\boldsymbol{\lambda}}\\
= &\ \hat{Q}_{\theta_{m}\wedge\hat{\sigma}^*_{k-1}\wedge\hat{\tau}^*_{k-1}}^{\boldsymbol{\lambda}}
\end{align*}
where the second last equality follows from the definition (\ref{optimalStop_aux}) of $(\hat{\sigma}^*_{k-1},\hat{\tau}^*_{k-1})$, and thus the martingale property (i) has been proved.

To prove the supermartingale property (ii), for any $\tau\in\tilde{\mathcal{R}}_{\theta_{k-1}}^{(2)}$, we have
\[\tilde{\mathbbm{E}}\left[\left.\hat{Q}_{\theta_{m+1}\wedge\hat{\sigma}^*_{k-1}\wedge\tau}^{\boldsymbol{\lambda}}\right|\mathcal{G}_{\theta_m}\right] = \mathbbm{1}_{\{\hat{\sigma}^*_{k-1}\wedge \tau \leq \theta_m\}} \hat{Q}_{\hat{\sigma}^*_{k-1}\wedge\tau}^{\boldsymbol{\lambda}} +\mathbbm{1}_{\{\hat{\sigma}^*_{k-1}\wedge \tau \geq \theta_{m+1}\}} \tilde{\mathbbm{E}}\bigg[\hat{Q}_{\theta_{m+1}}^{\boldsymbol{\lambda}}|\mathcal{G}_{\theta_m}\bigg].\]
Conditional on the set $\{\hat{\sigma}^*_{k-1}\wedge \tau \geq \theta_{m+1}\} \cap \{\theta_m<T\}$, we have
\begin{align*}
\hat{Q}_{\theta_{m}}^{\boldsymbol{\lambda}}=&\ \tilde{\mathbbm{E}}\left[\left.\hat{Q}^{\boldsymbol{\lambda}}_{\theta_{m+1}}\right|\mathcal{G}_{\theta_{m}}\right]\mathbbm{1}_{\{\theta_{m}\in T^{(1)}\}}+\max\left\{\tilde{L}_{\theta_{m}},\tilde{\mathbbm{E}}\left[\left.\hat{Q}^{\boldsymbol{\lambda}}_{\theta_{m+1}}\right|\mathcal{G}_{\theta_{m}}\right]\right\}\mathbbm{1}_{\{\theta_{m}\in T^{(2)}\}}\\
\geq &\ \tilde{\mathbbm{E}}\left[\left.\hat{Q}^{\boldsymbol{\lambda}}_{\theta_{m+1}}\right|\mathcal{G}_{\theta_{m}}\right],
\end{align*}
and thus
\begin{align*}
\tilde{\mathbbm{E}}\left[\left.\hat{Q}_{\theta_{m+1}\wedge\hat{\sigma}^*_{k-1}\wedge\tau}^{\boldsymbol{\lambda}}\right|\mathcal{G}_{\theta_m}\right]\leq &\ \mathbbm{1}_{\{\hat{\sigma}^*_{k-1}\wedge \tau \leq \theta_m\}} \hat{Q}_{\hat{\sigma}^*_{k-1}\wedge\tau}^{\boldsymbol{\lambda}}+\mathbbm{1}_{\{\hat{\sigma}^*_{k-1}\wedge \tau \geq \theta_{m+1}\}} \left(\tilde{\xi}\mathbbm{1}_{\{\theta_{m}  \geq T\}}+\hat{Q}_{\theta_{m}}^{\boldsymbol{\lambda}}\mathbbm{1}_{\{\theta_{m} < T\}}\right)\\
=&\ \hat{Q}_{\theta_{m}\wedge\hat{\sigma}^*_{k-1}\wedge\tau}^{\boldsymbol{\lambda}},
\end{align*}
which proves the supermartingale property (ii). Likewise, the submartingale property (iii) can be proved in a similar way, and the proof of this lemma is thus completed.
\end{proof}

We are now in a position to prove Theorem \ref{bigTheorem}. Let $\tilde{Q}_{\theta_{k-1}}^{\boldsymbol{\lambda}}$ be a solution of the recursive equation (\ref{recursiveEq}), and in turn,
\begin{align*}
\tilde{Q}_{\theta_{k-1}}^{\boldsymbol{\lambda}}=&\ \tilde{\mathbbm{E}}\left[\left.\tilde{\xi}\mathbbm{1}_{\{\theta_k  \geq T\}}+\hat{Q}_{\theta_k}^{\boldsymbol{\lambda}}\mathbbm{1}_{\{\theta_k < T\}}\right|\mathcal{G}_{\theta_{k-1}}\right]\\
=&\ \tilde{\mathbbm{E}}\left[\left.\tilde{\xi}\mathbbm{1}_{\{\theta_k  \geq T\}}+\tilde{\mathbbm{E}}\bigg[\tilde{R}(\hat{\sigma}^*_{k},\hat{\tau}^*_{k})|\mathcal{G}_{\theta_k}\bigg]\mathbbm{1}_{\{\theta_k < T\}}\right|\mathcal{G}_{\theta_{k-1}}\right]\\
=&\ \tilde{\mathbbm{E}}\left[\left.\tilde{\mathbbm{E}}\bigg[\tilde{\xi}\mathbbm{1}_{\{\theta_k  \geq T\}}+\tilde{R}(\hat{\sigma}^*_{k},\hat{\tau}^*_{k})\mathbbm{1}_{\{\theta_k < T\}}|\mathcal{G}_{\theta_k}\bigg]\right|\mathcal{G}_{\theta_{k-1}}\right]\\
=&\  \tilde{\mathbbm{E}}\bigg[\tilde{\xi}\left(\mathbbm{1}_{\{\theta_k  \geq T\}}+\mathbbm{1}_{\{\hat{\sigma}^*_{k}\wedge\hat{\tau}^*_{k} \geq T\}}\mathbbm{1}_{\{\theta_k < T\}}\right)+\tilde{L}_{\hat{\tau}^*_{k}}\mathbbm{1}_{\{\hat{\tau}^*_{k} < T,\hat{\tau}^*_{k} \leq \hat{\sigma}^*_{k}\}}\mathbbm{1}_{\{\theta_k < T\}}\\
&+\tilde{U}_{\hat{\sigma}^*_{k}}\mathbbm{1}_{\{\hat{\sigma}^*_{k}< T,\hat{\sigma}^*_{k} < \hat{\tau}^*_{k} \}}\mathbbm{1}_{\{\theta_k < T\}}|\mathcal{G}_{\theta_{k-1}}\bigg].
\end{align*}
Using the relationship $\{\theta_k  \geq T\} \subseteq \{\hat{\sigma}^*_{k}\wedge\hat{\tau}^*_{k} \geq T\}$, $\{\hat{\tau}^*_{k} < T,\hat{\tau}^*_{k} \leq \hat{\sigma}^*_{k}\} \subseteq \{\theta_k < T\}$ and $\{\hat{\sigma}^*_{k}< T,\hat{\sigma}^*_{k} < \hat{\tau}^*_{k} \} \subseteq \{\theta_k < T\}$, we can further obtain that
\begin{align}
\tilde{Q}_{\theta_{k-1}}^{\boldsymbol{\lambda}}=&\ \tilde{\mathbbm{E}}\bigg[\tilde{\xi}\mathbbm{1}_{\{\hat{\sigma}^*_{k}\wedge\hat{\tau}^*_{k} \geq T\}}+\tilde{L}_{\hat{\tau}^*_{k}}\mathbbm{1}_{\{\hat{\tau}^*_{k} < T,\hat{\tau}^*_{k} \leq \hat{\sigma}^*_{k}\}}+\tilde{U}_{\hat{\sigma}^*_{k}}\mathbbm{1}_{\{\hat{\sigma}^*_{k}< T,\hat{\sigma}^*_{k} < \hat{\tau}^*_{k} \}}|\mathcal{G}_{\theta_{k-1}}\bigg]\nonumber\\
=&\ \tilde{\mathbbm{E}}\bigg[\tilde{R}(\hat{\sigma}^*_{k},\hat{\tau}^*_{k})|\mathcal{G}_{\theta_{k-1}}\bigg],\label{main_proof_general_game_martingale_eq01}
\end{align}
where $(\hat{\sigma}^*_{k},\hat{\tau}^*_{k})$ is the optimal stopping strategy of the auxiliary constrained risk-sensitive Dynkin game starting from $\theta_k$ given in (\ref{optimalStop_aux}). Similarly, we can obtain that, for any $\tau\in\tilde{\mathcal{R}}_{\theta_{k}}^{(2)}$,
\begin{equation}
\label{main_proof_general_game_martingale_eq02}
\tilde{Q}_{\theta_{k-1}}^{\boldsymbol{\lambda}} \geq \tilde{\mathbbm{E}}\bigg[\tilde{R}(\hat{\sigma}^*_{k},\tau)|\mathcal{G}_{\theta_{k-1}}\bigg],
\end{equation}
and, for any $\sigma\in\tilde{\mathcal{R}}_{\theta_{k}}^{(1)}$,
\begin{equation}
\label{main_proof_general_game_martingale_eq03}
\tilde{Q}_{\theta_{k-1}}^{\boldsymbol{\lambda}} \leq \tilde{\mathbbm{E}}\bigg[\tilde{R}(\sigma,\hat{\tau}^*_{k})|\mathcal{G}_{\theta_{k-1}}\bigg].
\end{equation}

It follows from (\ref{main_proof_general_game_martingale_eq02}) and (\ref{main_proof_general_game_martingale_eq03}) that
\[\tilde{Q}_{\theta_{k-1}}^{\boldsymbol{\lambda}} \geq \esssup_{\tau\in\tilde{\mathcal{R}}_{\theta_{k}}^{(2)}}\tilde{\mathbbm{E}}\bigg[\tilde{R}(\hat{\sigma}^*_{k},\tau)|\mathcal{G}_{\theta_{k-1}}\bigg]\geq \essinf_{\sigma\in\tilde{\mathcal{R}}_{\theta_{k}}^{(1)}}\esssup_{\tau\in\tilde{\mathcal{R}}_{\theta_{k}}^{(2)}}\tilde{\mathbbm{E}}\bigg[\tilde{R}(\sigma,\tau)|\mathcal{G}_{\theta_{k-1}}\bigg] =\overline{q}_{\theta_{k-1}}^{\boldsymbol{\lambda}},\]
and
\[\tilde{Q}_{\theta_{k-1}}^{\boldsymbol{\lambda}} \leq \essinf_{\sigma\in\tilde{\mathcal{R}}_{\theta_{k}}^{(1)}}\tilde{\mathbbm{E}}\bigg[\tilde{R}(\sigma,\hat{\tau}^*_{k})|\mathcal{G}_{\theta_{k-1}}\bigg]\leq \esssup_{\tau\in\tilde{\mathcal{R}}_{\theta_{k}}^{(2)}}\essinf_{\sigma\in\tilde{\mathcal{R}}_{\theta_{k}}^{(1)}}\tilde{\mathbbm{E}}\bigg[\tilde{R}(\sigma,\tau)|\mathcal{G}_{\theta_{k-1}}\bigg] = \underline{q}_{\theta_{k-1}}^{\boldsymbol{\lambda}}.\]

It is clear that $\overline{q}_{\theta_{k-1}}^{\boldsymbol{\lambda}} \geq \underline{q}_{\theta_{k-1}}^{\boldsymbol{\lambda}}$, and therefore the value of the constrained risk-sensitive Dynkin game starting from $\theta_{k-1}$ (\ref{upper_Values_discounted})-(\ref{lower_Values_discounted}) exists, i.e.
\[\tilde{Q}_{\theta_{k-1}}^{\boldsymbol{\lambda}}=q_{\theta_{k-1}}^{\boldsymbol{\lambda}}=\overline{q}_{\theta_{k-1}}^{\boldsymbol{\lambda}}=\underline{q}_{\theta_{k-1}}^{\boldsymbol{\lambda}}.\]
This also implies the recursive equation (\ref{recursiveEq}) admits a unique solution. Furthermore, since $\tilde{Q}_{\theta_{k-1}}^{\boldsymbol{\lambda}}$ given by (\ref{final_game_value_general_discounted}) satisfies the recursive equation (\ref{recursiveEq}), it is actually the unique solution of (\ref{recursiveEq}). As a direct consequence of (\ref{main_proof_general_game_martingale_eq01})-(\ref{main_proof_general_game_martingale_eq03}), we can obtain that $(\hat{\sigma}^*_{k},\hat{\tau}^*_{k})$, which is given by (\ref{optimalStop_aux}), is indeed an optimal stopping strategy of the constrained risk-sensitive Dynkin game (\ref{upper_Values_discounted})-(\ref{lower_Values_discounted}).

We conclude the proof by proving $(\hat{\sigma}^*_{k},\hat{\tau}^*_{k})$ are actually $(\sigma^*_{k},\tau^*_{k})$ in (\ref{optimalStop_general}). Indeed,
\begin{eqnarray*}
\hat{\sigma}^*_{k}&=&\inf\{T^{(1)}_N\geq \theta_{k}:\hat{Q}_{T^{(1)}_N}^{\boldsymbol{\lambda}}= \tilde{U}_{T^{(1)}_N}\}\wedge T^{(1)}_{M_1}\\
&=&\inf\{T^{(1)}_N\geq \theta_{k}:\tilde{Q}_{T^{(1)}_N}^{\boldsymbol{\lambda}} \geq \tilde{U}_{T^{(1)}_N}\}\wedge T^{(1)}_{M_1}=\sigma^*_k,
\end{eqnarray*}
and, similarly, $\hat{\tau}^*_{k} = \tau^*_{k}$.

%%%%%%%%%%%%%%%%%%%%%%%%%%%%%%%%%%%%%%%%%%%%%%%%%%%%%%%
\section{Connection with stochastic differential games via randomized stopping}
In this section, we connect constrained risk-sensitive Dynkin games
with a class of stochastic differential games via randomized stopping first introduced by Krylov (see \cite{Krylov}).
In particular, we generalize the optimal control representation of constrained
optimal stopping problems in \cite{liang2015stochastic} (see section
4 therein).

Let us introduce the basic idea of randomized stopping in a two-player setting as follows. Consider a nonnegative control process $(a_t)_{t\geq 0}$ (resp. $(b_t)_{t\geq 0}$), and let Player I (resp. II) stop with probability $a_t\Delta$ (resp. $b_t\Delta$) in an infinitesimal interval $(t,t+\Delta)$. Then the probability that Player I (resp. II) does not stop before time $t$ is
\[e^{-\int_0^t a_u\,du}~~\left(\mbox{resp. }e^{-\int_0^t b_u\,du}\right),\]
and the probability that both players do not stop before time $t$ and Player I (resp. II) does stop in the infinitesimal interval $(t,t+\Delta)$ is
\[e^{-\int_0^t (a_u+b_u)\,du}a_t\Delta~~\left(\mbox{resp. }e^{-\int_0^t (a_u+b_u)\,du}b_t\Delta\right).\]

Recall that $T$ is a finite $\mathbbm{F}$-stopping time representing the (random) terminal time of the game, and $r>0$ represents the discount rate. The discounted payoff is assumed to be $e^{-rt}\overline{U}_t$ if Player I stops firstly
at time $t<T$, $e^{-rt}\overline{L}_t$ if Player II stops firstly
at time $t<T$, and $e^{-rT}\overline{\xi}$ if neither players stop in the time interval $[0,T]$, where the auxiliary payoff processes $\overline{U}$, $\overline{L}$ and $\overline{\xi}$ are given in (\ref{final_game_value_parameters_U}), (\ref{final_game_value_parameters_L}), and (\ref{final_game_value_parameters_xi}), respectively. Thus, the discounted payoff functional associated with the control processes $a$ and $b$ is given by
\[J(a,b) = \int_0^T e^{-\int_0^t (a_u+b_u+r)\,du}\left(a_t\overline{U}_t+b_t\overline{L}_t\right)\,dt+e^{-\int_0^T (a_u+b_u+r)\,du}\overline{\xi},\]
or in terms of the original processes $L$, $U$ and $\xi$,
\begin{align*}
J(a,b) =&\ \int_0^T e^{-\int_0^t (a_u+b_u)\,du}\left[a_tg(e^{-rt}U_t+\int_0^te^{-ru}f_udu)+b_tg(e^{-rt}L_t+\int_0^te^{-ru}f_udu)\right]\\
&+e^{-\int_0^T (a_u+b_u)\,du}g(e^{-rT}\xi+\int_0^Te^{-ru}f_udu).
\end{align*}

Let us define the control set $\mathcal{A}(\lambda^{(1)})$ (resp. $\mathcal{B}(\lambda^{(2)})$) for Player I (resp. II) as
\[\mathcal{A}(\lambda^{(1)}) = \{\mathbbm{F}\mbox{-adapted process }(a_t)_{t \geq 0}:a_t=0\mbox{ or }\lambda^{(1)}\}\]
(resp.
\[\mathcal{B}(\lambda^{(2)}) = \{\mathbbm{F}\mbox{-adapted process }(b_t)_{t \geq 0}:b_t=0\mbox{ or }\lambda^{(2)}\}),\]
and the upper and lower values of the stochastic differential game as
\begin{equation}
\label{upper_lower_Values_SDG}
\overline{v}^{\boldsymbol{\lambda},SDG}=\inf_{a\in \mathcal{A}(\lambda^{(1)})}\sup_{b\in\mathcal{B}(\lambda^{(2)})}g^{-1}\left(\mathbbm{E}[J(a,b)]\right),~~\underline{v}^{\boldsymbol{\lambda},SDG}=\sup_{b\in\mathcal{B}(\lambda^{(2)})}\inf_{a\in \mathcal{A}(\lambda^{(1)})}g^{-1}\left(\mathbbm{E}[J(a,b)]\right),
\end{equation}
where $g^{-1}$ is the inverse function of the risk-sensitive function $g$. The game (\ref{upper_lower_Values_SDG}) is said to have value $v^{\boldsymbol{\lambda},SDG}$ if $v^{\boldsymbol{\lambda},SDG}=\overline{v}^{\boldsymbol{\lambda},SDG}=\underline{v}^{\boldsymbol{\lambda},SDG}$, and $(a^*,b^*)\in \mathcal{A}(\lambda^{(1)}) \times \mathcal{B}(\lambda^{(2)})$ is said to be an optimal pair of controls if $v^{\boldsymbol{\lambda},SDG}=g^{-1}\left(\mathbbm{E}[J(a^*,b^*)]\right)$.

We are now in a position to present the main result of this section.

\begin{proposition}
Suppose that Assumption \ref{assumption_1} holds. Let $(\overline{Q}^{\boldsymbol{\lambda}},\overline{Z}^{\boldsymbol{\lambda}})$ be the unique solution to BSDE (\ref{final_game_value}). Then, the value of the stochastic differential game (\ref{upper_lower_Values_SDG}) exists and equals the value $v^{\boldsymbol{\lambda}}$ of the constrained risk-sensitive Dynkin game (\ref{upper_lowerValues}), i.e.
\begin{equation}
\label{prop_connection_SDG_eq01}
v^{\boldsymbol{\lambda},SDG}=\overline{v}^{\boldsymbol{\lambda},SDG}=\underline{v}^{\boldsymbol{\lambda},SDG}=v^{\boldsymbol{\lambda}}=g^{-1}\left(\overline{Q}^{\boldsymbol{\lambda}}_0\right).
\end{equation}
Moreover, the optimal pair of controls is given by
\begin{equation}
\label{prop_connection_SDG_eq02}
a_t^*=\lambda^{(1)}\mathbbm{1}_{\{\overline{Q}_t^{\boldsymbol{\lambda}} \geq \overline{U}_t\}},~~b_t^*=\lambda^{(2)}\mathbbm{1}_{\{\overline{Q}_t^{\boldsymbol{\lambda}} \leq \overline{L}_t\}}
\end{equation}
for $t \geq 0$.
\end{proposition}

\begin{proof}
Following the similar arguments to the proof of Lemma \ref{firstLemma}, it can be shown that, for any pair of controls $(a,b)\in \mathcal{A}(\lambda^{(1)}) \times \mathcal{B}(\lambda^{(2)})$, $\mathbbm{E}[J(a,b)]=V^{\boldsymbol{\lambda}}_0(a,b)$, where the latter is the first component of the unique solution to the following BSDE with a random terminal time $T$:
\[V^{\boldsymbol{\lambda}}_{t\wedge T}(a,b) =\overline{\xi} + \int_{t\wedge T}^T \bigg[a_u(\overline{U}_u-V^{\boldsymbol{\lambda}}_u(a,b))+b_u(\overline{L}_u-V^{\boldsymbol{\lambda}}_u(a,b))-rV^{\boldsymbol{\lambda}}_u(a,b)\bigg]\,du-\int_{t\wedge T}^TZ^{\boldsymbol{\lambda}}_u(a,b)\,dW_u,\]
for $t \geq 0$. On the other hand, recall that $\overline{Q}^{\boldsymbol{\lambda}}$ is the first component of the solution to BSDE (\ref{final_game_value}):
\[\overline{Q}_{t\wedge T}^{\boldsymbol{\lambda}}=\overline{\xi}+ \int_{t\wedge T}^T\left[-\lambda^{(1)} \left(\overline{Q}_u^{\boldsymbol{\lambda}}-\overline{U}_u\right)^+ + \lambda^{(2)}\left(\overline{L}_u-\overline{Q}_u^{\boldsymbol{\lambda}} \right)^+-r\overline{Q}_{u}^{\boldsymbol{\lambda}}\right]\,du - \int_{t\wedge T}^T\overline{Z}_u^{\boldsymbol{\lambda}}\,dW_u,\]
for $t \geq 0$. By letting $b_t^*=\lambda^{(2)}\mathbbm{1}_{\{\overline{Q}_{t}^{\boldsymbol{\lambda}}\leq \overline{L}_t\}}$, we obtain the inequality
\[-\lambda^{(1)} (\overline{Q}_u^{\boldsymbol{\lambda}}-\overline{U}_u)^+ + \lambda^{(2)}(\overline{L}_u-\overline{Q}_u^{\boldsymbol{\lambda}})^+-r\overline{Q}_{u}^{\boldsymbol{\lambda}}\leq  a_u (\overline{U}_u-\overline{Q}_u^{\boldsymbol{\lambda}})+b^*_u(\overline{L}_u-\overline{Q}_u^{\boldsymbol{\lambda}})-r\overline{Q}_{u}^{\boldsymbol{\lambda}}\]
holds for any control $a\in \mathcal{A}(\lambda^{(1)})$, and thus, the BSDE comparison result (see Corollary 4.4.2 in \cite{darling1997backwards}) yields that
\begin{equation}
\label{connection_SDG_eq01}
\overline{Q}^{\boldsymbol{\lambda}}_{t \wedge T} \leq V^{\boldsymbol{\lambda}}_{t \wedge T}(a,b^*),
\end{equation}
for $t \geq 0$ and any control $a\in \mathcal{A}(\lambda^{(1)})$. Similarly, by letting $a_t^*=\lambda^{(1)}\mathbbm{1}_{\{\overline{Q}_{t}^{\boldsymbol{\lambda}} \geq \overline{U}_t\}}$, we obtain
\begin{equation}
\label{connection_SDG_eq02}
\overline{Q}^{\boldsymbol{\lambda}}_{t \wedge T} \geq V^{\boldsymbol{\lambda}}_{t \wedge T}(a^*,b),
\end{equation}
for $t \geq 0$ and any control $b\in \mathcal{B}(\lambda^{(2)})$, and by letting $a_t^*=\lambda^{(1)}\mathbbm{1}_{\{\overline{Q}_{t}^{\boldsymbol{\lambda}} \geq \overline{U}_t\}}$ and $b_t^*=\lambda^{(2)}\mathbbm{1}_{\{\overline{Q}_{t}^{\boldsymbol{\lambda}} \leq \overline{L}_t\}}$, we obtain the equality
\begin{equation}
\label{connection_SDG_eq03}
\overline{Q}^{\boldsymbol{\lambda}}_{t \wedge T} = V^{\boldsymbol{\lambda}}_{t \wedge T}(a^*,b^*).
\end{equation}
It follows from (\ref{connection_SDG_eq01}) that
\begin{eqnarray*}
g^{-1}\left(\overline{Q}^{\boldsymbol{\lambda}}_0\right) \leq \inf_{a\in \mathcal{A}(\lambda^{(1)})}g^{-1}\left(V^{\boldsymbol{\lambda}}_0(a,b^*)\right) &=& \inf_{a\in \mathcal{A}(\lambda^{(1)})}g^{-1}\left(\mathbbm{E}[J(a,b^*)]\right) \\
&\leq& \sup_{b\in\mathcal{B}(\lambda^{(2)})}\inf_{a\in \mathcal{A}(\lambda^{(1)})}g^{-1}\left(\mathbbm{E}[J(a,b)]\right) = \underline{v}^{\boldsymbol{\lambda},SDG}.
\end{eqnarray*}
Likewise, (\ref{connection_SDG_eq02}) yields that $g^{-1}\left(\overline{Q}^{\boldsymbol{\lambda}}_0\right) \geq \overline{v}^{\boldsymbol{\lambda},SDG}$. Hence, it follows from $\overline{v}^{\boldsymbol{\lambda},SDG} \geq \underline{v}^{\boldsymbol{\lambda},SDG}$ that (\ref{prop_connection_SDG_eq01}) holds. As a direct consequence of (\ref{connection_SDG_eq01})-(\ref{connection_SDG_eq03}), we can obtain $(a^*,b^*)$ in (\ref{prop_connection_SDG_eq02}) is an optimal pair of controls.
\end{proof}

%%%%%%%%%%%%%%%%%%%%%%%%%%%%%%%%%%%%%%%%%%%%%%%%%%%%%%%
\section{Examples}
\subsection{Example I: Constrained risk-neutral Dynkin games}
\label{section_risk_neutral} As the first example, we take the
risk-sensitive function to be $g(x)=x$. This means both players are
risk neutral and, therefore, the corresponding games are called
\emph{constrained risk-neutral Dynkin games}. In this case, the cost
functional in (\ref{risk_sensitive_cost_functional}) is evaluated
under the linear expectation $\mathbbm{E}$:
\[\tilde{\mathbbm{E}}\left[R(\sigma,\tau)\right] = \mathbbm{E}\left[R(\sigma,\tau)\right]\]
with the payoff functional $R(\sigma,\tau)$ given by
(\ref{def_payoff}). Hence, the upper and lower values of the
constrained risk-neutral Dynkin game are defined as
\begin{equation}
\label{upper_lowerValues_risk_neutral}
\overline{v}^{\boldsymbol{\lambda},RN}=\inf_{\sigma\in\mathcal{R}_1^{(1)}}\sup_{\tau\in\mathcal{R}_1^{(2)}}\mathbbm{E}[R(\sigma,\tau)],\mbox{ and }\underline{v}^{\boldsymbol{\lambda},RN}=\sup_{\tau\in\mathcal{R}_1^{(2)}}\inf_{\sigma\in\mathcal{R}_1^{(1)}}\mathbbm{E}[R(\sigma,\tau)].
\end{equation}
The game (\ref{upper_lowerValues_risk_neutral}) is said to have value $v^{\boldsymbol{\lambda},RN}$ if $v^{\boldsymbol{\lambda},RN}=\overline{v}^{\boldsymbol{\lambda},RN}=\underline{v}^{\boldsymbol{\lambda},RN}$, and $(\sigma^{*,RN},\tau^{*,RN})\in \mathcal{R}_1^{(1)}\times \mathcal{R}_1^{(2)}$ is called an optimal stopping strategy of the game if
\[\mathbbm{E}\left[R(\sigma^{*,RN},\tau)\right] \leq \mathbbm{E}\left[R(\sigma^{*,RN},\tau^{*,RN})\right] \leq \mathbbm{E}\left[R(\sigma,\tau^{*,RN})\right]\]
for every $(\sigma,\tau)\in \mathcal{R}_1^{(1)}\times \mathcal{R}_1^{(2)}$.

%Using the definitions of $\overline{L}$, $\overline{U}$ and
%$\overline{\xi}$ in
%(\ref{final_game_value_parameters_L})-(\ref{final_game_value_parameters_xi})
%and assumption $g(x)=x$, we reduce the characterizing BSDE
%(\ref{final_game_value}) to
%\begin{align*}
%&e^{-r(t\wedge T)}\overline{Q}_{t\wedge
%T}^{\boldsymbol{\lambda}}-\int_0^{t\wedge T}e^{-ru}f_udu
%=e^{-rT}{\xi}\mathbbm{1}_{\{T<\infty\}} - \int_{t\wedge T}^Te^{-rs}\overline{Z}_s^{\boldsymbol{\lambda}}\,dW_s\\
%&+ \int_{t\wedge T}^T\left[e^{-rs}f_s-\lambda^{(1)}
%\left(e^{-rs}(\overline{Q}_s^{\boldsymbol{\lambda}}-U_s)-\int_0^{s}e^{-ru}f_udu\right)^+
%+
%\lambda^{(2)}\left(e^{-rs}({L}_s-\overline{Q}_s^{\boldsymbol{\lambda}})+\int_0^{s}e^{-ru}f_udu
%\right)^+\right]\,ds.
%\end{align*}
Recall
\[Q_t^{\boldsymbol{\lambda}} =
\overline{Q}^{\boldsymbol{\lambda}}_t-\int_0^{t\wedge T} e^{-r(u-t\wedge T)}f_u\,du\]
in (\ref{value_process}), where $(\overline{Q}^{\boldsymbol{\lambda}},\overline{Z}^{\boldsymbol{\lambda}})$ is the unique solution to the characterizing BSDE (\ref{final_game_value}). Thus, we deduce the so-called penalized BSDE with
double obstacles on a random horizon $[0,T]$ (see \cite{cvitanic1996backward} for the case of a
fixed terminal time $T$),
\begin{equation}\label{final_game_value_risk_neutral}
Q_{t\wedge T}^{\boldsymbol{\lambda}}={\xi}+\int_{t\wedge
T}^T\left[f_s-\lambda^{(1)}
\left(Q_s^{\boldsymbol{\lambda}}-{U}_s\right)^+ +
\lambda^{(2)}\left({L}_s-Q_s^{\boldsymbol{\lambda}}
\right)^+-rQ^{\boldsymbol{\lambda}}_s\right]\,ds-\int_{t\wedge
T}^T\overline{Z}_s^{\boldsymbol{\lambda}}\,dW_s,
\end{equation}
and $Q_{t}^{\boldsymbol{\lambda}}=\overline{\xi}-\int_0^{T} e^{-r(u-T)}f_u\,du=\xi$ for $t\geq T$.
\begin{assumption}
\label{assumption_2} The risk-sensitive function $g(x)=x$. Moreover,
(i) when $T$ is an unbounded stopping time, $f$, $L$, $U$ and $\xi$ are all bounded; (ii) when $T$ is a bounded stopping time, $f\in \mathbbm{H}^2_1$, $L\in\mathbbm{S}^2_1$, $U\in\mathbbm{S}^2_1$ and $\xi\in\mathbbm{L}^2_1$.
\end{assumption}

Note that the above assumption implies Assumption \ref{assumption_1}
and, therefore, it follows from Theorem \ref{bigTheorem} that BSDE
$(\ref{final_game_value_risk_neutral})$ admits a unique solution
$(Q^{\boldsymbol{\lambda}},\overline{Z}^{\boldsymbol{\lambda}})$.
Moreover, the value of the constrained risk-neutral Dynkin game
(\ref{upper_lowerValues_risk_neutral}) exists and is given by
\[v^{\boldsymbol{\lambda},RN}=\overline{v}^{\boldsymbol{\lambda},RN}=\underline{v}^{\boldsymbol{\lambda},RN}=Q_0^{\boldsymbol{\lambda}}.\] The optimal stopping strategy is given by
\begin{equation*}
\left\{\begin{array}{l}
\sigma^{*,RN}=\inf\{T_N^{(1)}\geq T^{(1)}_1:Q_{T_N^{(1)}}^{\boldsymbol{\lambda}}\geq U_{T_N^{(1)}}\}\wedge T^{(1)}_{M_1};\\
\tau^{*,RN}=\inf\{T_N^{(2)}\geq
T^{(2)}_1:Q_{T_N^{(2)}}^{\boldsymbol{\lambda}}\leq
L_{T_N^{(2)}}\}\wedge T^{(2)}_{M_2}.
\end{array}\right.
\end{equation*}

\begin{remark} The special case $g(x)=x$ generalizes the
results obtained in \cite{liang2015stochastic} and
\cite{liang2018dynkin}. To be more specific, when $\lambda^{(1)}=0$
(resp. $\lambda^{(2)}=0$), Player I (resp. II) is with a zero intensity
control set and is never allowed to stop, so the value of the
constrained risk-neutral Dynkin game
(\ref{upper_lowerValues_risk_neutral}) equals to the value of the
one-player optimal stopping problem with Poisson intervention times
introduced in \cite{liang2015stochastic}. On the other hand, when
the two intensities coincide, i.e. $\lambda^{(1)}=\lambda^{(2)}$,
the value of the constrained risk-neutral Dynkin game
(\ref{upper_lowerValues_risk_neutral}) equals to the value of the
Dynkin game with Poisson intervention times introduced in
\cite{liang2018dynkin}.
\end{remark}

\subsection{Example II: Constrained Dynkin games with exponential utility}\label{section:exponential}

The second example for the risk-sensitive function $g$ is an
exponential utility: $g(x)=-e^{-\gamma x}$ for $\gamma>0$. In this case, the cost
functional in (\ref{risk_sensitive_cost_functional}) becomes
\[\tilde{\mathbbm{E}}\left[R(\sigma,\tau)\right]
=-\frac{1}{\gamma}\ln
\mathbbm{E}\left[\exp(-\gamma R(\sigma,\tau))\right]\] with the payoff
functional $R(\sigma,\tau)$ given by (\ref{def_payoff}). Hence, the
upper and lower values of the constrained risk-sensitive Dynkin game
are defined as
\begin{eqnarray}
\overline{v}^{\boldsymbol{\lambda},EU} &=& \inf_{\sigma\in\mathcal{R}_1^{(1)}}\sup_{\tau\in\mathcal{R}_1^{(2)}}-\frac{1}{\gamma}\ln\mathbbm{E}[\exp(-\gamma R(\sigma,\tau))],\label{upper_Values_exponential_utility}\\
\underline{v}^{\boldsymbol{\lambda},EU} &=& \sup_{\tau\in\mathcal{R}_1^{(2)}}\inf_{\sigma\in\mathcal{R}_1^{(1)}}-\frac{1}{\gamma}\ln\mathbbm{E}[\exp(-\gamma R(\sigma,\tau))].\label{lower_Values_exponential_utility}
\end{eqnarray}
The game (\ref{upper_Values_exponential_utility})-(\ref{lower_Values_exponential_utility}) is said to
have value $v^{\boldsymbol{\lambda},EU}$ if
$v^{\boldsymbol{\lambda},EU}=\overline{v}^{\boldsymbol{\lambda},EU}=\underline{v}^{\boldsymbol{\lambda},EU}$,
and $(\sigma^{*,EU},\tau^{*,EU})\in \mathcal{R}_1^{(1)}\times
\mathcal{R}_1^{(2)}$ is called an optimal stopping strategy of the
game if
\[\tilde{\mathbbm{E}}\left[R(\sigma^{*,EU},\tau)\right] \leq \tilde{\mathbbm{E}}\left[R(\sigma^{*,EU},\tau^{*,EU})\right] \leq \tilde{\mathbbm{E}}\left[R(\sigma,\tau^{*,EU})\right]\]
for every $(\sigma,\tau)\in
\mathcal{R}_1^{(1)}\times \mathcal{R}_1^{(2)}$.

Recall
\[Q_t^{\boldsymbol{\lambda}} =-\frac{1}{\gamma}e^{r(t\wedge T)}\ln(-e^{-r(t\wedge T)}\overline{Q}^{\boldsymbol{\lambda}}_t)-\int_0^{t \wedge T} e^{-r(u-t\wedge T)}f_u\,du\]
in (\ref{value_process}), where $(\overline{Q}^{\boldsymbol{\lambda}},\overline{Z}^{\boldsymbol{\lambda}})$ is the unique solution to the characterizing BSDE (\ref{final_game_value}). Thus, we deduce the following BSDE with quadratic growth on a random horizon $[0, T]$ (see \cite{Kobylanski} for the case of a fixed maturity $T$):
\begin{align}
Q_{t\wedge T}^{\boldsymbol{\lambda}}=&\ \xi+\int_{t \wedge T}^T \bigg[f_u-\frac{\lambda^{(1)}}{\gamma}e^{ru}(e^{\gamma (e^{-ru}Q^{\boldsymbol{\lambda}}_u-e^{-ru}U_u)}-1)^++\frac{\lambda^{(2)}}{\gamma}e^{ru}(1-e^{\gamma (e^{-ru}Q^{\boldsymbol{\lambda}}_u-e^{-ru}L_u)})^+\nonumber\\
&-rQ^{\boldsymbol{\lambda}}_u-\frac{\gamma}{2}e^{-ru}||Z^{\boldsymbol{\lambda}}_u||^2\bigg]\,du-\int_{t \wedge T}^TZ^{\boldsymbol{\lambda}}_u\,dW_u,\label{quadratic_BSDE}
\end{align}
for $t\geq 0$, where $Z^{\boldsymbol{\lambda}}_u=-e^{ru}\overline{Z}_u^{\boldsymbol{\lambda}}/(\gamma\overline{Q}_u^{\boldsymbol{\lambda}}), u\in[0,T]$. Note that, for $t\geq T$, $$Q_t^{\boldsymbol{\lambda}} =-\frac{1}{\gamma}e^{r T}\ln(-e^{-rT}\overline{\xi})-\int_0^{T} e^{-r(u-T)}f_u\,du=\xi.$$
%With $g(x)=-e^{-\gamma x}$ and a fixed maturity $T>0$, we first
%rewrite the characterizing BSDE (\ref{final_game_value}) in terms of
%$g(\tilde{Q}^{\boldsymbol{\lambda}}_t)=e^{-rt}\overline{Q}^{\boldsymbol{\lambda}}_t$,
%$$g(\tilde{Q}^{\boldsymbol{\lambda}}_t)=g(\tilde{\xi})+\int_t^T\left[-\lambda^{(1)}\left(g(\tilde{Q}^{\boldsymbol{\lambda}}_s)-g(\tilde{U}_s)\right)^++
%\lambda^{(2)}\left(g(\tilde{L}_s)-g(\tilde{Q}^{\boldsymbol{\lambda}}_s)\right)^+\right]ds-\int_t^Te^{-rs}\overline{Z}^{\boldsymbol{\lambda}}_sdW_s.$$
%In turn, It\^o's formula implies that the following BSDE with
%quadratic growth (see \cite{Kobylanski}),
%\begin{equation}
%\tilde{Q}^{\boldsymbol{\lambda}}_t=\tilde{\xi}+\int_t^T\left[\frac{-\lambda^{(1)}}{\gamma}\left(1-e^{\gamma(\tilde{Q}^{\boldsymbol{\lambda}}_t-
%\tilde{U}_t)}\right)^++
%\frac{\lambda^{(2)}}{\gamma}\left(e^{\gamma(\tilde{Q}^{\boldsymbol{\lambda}}_s-\tilde{L}_s)}-1\right)^++\frac{\gamma}{2}||\tilde{Z}_s^{\boldsymbol{\lambda}}||^2\right]ds-\int_t^T\tilde{Z}_s^{\boldsymbol{\lambda}}dW_s,
%\end{equation}
%where
%$\tilde{Z}_t^{\boldsymbol{\lambda}}=\overline{Z}_t^{\boldsymbol{\lambda}}/(\gamma
%\overline{Q}_t^{\boldsymbol{\lambda}})$ for $t\in[0,T]$.

\begin{assumption}
\label{assumption_3}
The risk-sensitive function $g(x)=-e^{-\gamma
x}$ for $\gamma>0$, and
$f$, $L$, $U$ and $\xi$ are all bounded.
\end{assumption}

Note that the above assumption implies Assumption \ref{assumption_1}
and, therefore, it follows from Theorem \ref{bigTheorem} that BSDE
(\ref{quadratic_BSDE}) admits a unique solution
$(Q^{\boldsymbol{\lambda}},Z^{\boldsymbol{\lambda}})$.
Moreover, the value of the constrained risk-sensitive Dynkin game
(\ref{upper_Values_exponential_utility})-(\ref{lower_Values_exponential_utility}) exists and is given by
\[v^{\boldsymbol{\lambda},EU}=\overline{v}^{\boldsymbol{\lambda},EU}=\underline{v}^{\boldsymbol{\lambda},EU}=Q^{\boldsymbol{\lambda}}_0.\] The optimal stopping strategy is given by
\begin{equation*}
\left\{\begin{array}{l}
\sigma^{*,EU}=\inf\{T_N^{(1)}\geq T^{(1)}_1:Q_{T_N^{(1)}}^{\boldsymbol{\lambda}}\geq U_{T_N^{(1)}}\}\wedge T^{(1)}_{M_1};\\
\tau^{*,EU}=\inf\{T_N^{(2)}\geq
T^{(2)}_1:Q_{T_N^{(2)}}^{\boldsymbol{\lambda}}\leq L_{T_N^{(2)}}\}\wedge T^{(2)}_{M_2}.
\end{array}\right.
\end{equation*}

\section{Conclusions}

In this paper, we have solved a new class of Dynkin games with a
general risk-sensitive criterion function $g$ and two heterogenous
Poisson arrival times as the permitted stopping time strategies for the two
players. Moreover, we have made a connection with a class of
stochastic differential games via the so-called randomized stopping
technique.

The approach and the results herein may be extended in various directions. First,
one may consider stochastic intensity models, an undoubtedly important
case since the two players' signal times may affect each other's intensities. For example, for $i\in\{1,2\}$, if the player $i$'s first signal time $T^{(i)}_1$ occurs, it will have {an impact (either positive or negative)} on the other player $(3-i)$'s intensity:
$$\lambda^{(1)}_t=\lambda^{(1)}+\overline{\lambda}^{(1)}\mathbbm{1}_{\{T^{(2)}_1\leq t\}},\quad
\lambda^{(2)}_t =\lambda^{(2)}+\overline{\lambda}^{(2)}\mathbbm{1}_{\{T^{(1)}_1\leq t\}},$$
{for some constants $\lambda^{(i)}, \overline{\lambda}^{(i)}$ such that the process $(\lambda^{(i)}_t)_{t \geq 0}$ is always nonnegative.} However, various nontrivial technical difficulties
arise. In particular, the resulting characterizing BSDEs will become a family of recursive equations, whose solvability is far from clear yet.

Second, one may consider that the two players have different attitudes towards risks and are associated with different information sets.
For example, one player is risk-neutral with $g^{(1)}(x)=x$ and the other has an exponential utility with $g^{(2)}(x)=-e^{-\gamma x}$. This leads to heterogenous payoff functionals
and, therefore a nonzero-sum constrained Dynkin game arises. The corresponding characterizing equations will become a BSDE system. Both extensions will be left for the future research.

%%%%%% References%%%%%%%%%%%%%%%%%%%%%%%%%%%%%%%%%%%%%

\end{document}